\documentclass[12pt,a4paper,oneside]{article}
\usepackage{amsmath,amsfonts,amssymb,latexsym}

\newtheorem{theorem}{Theorem}[section]
\newtheorem{lemma}[theorem]{Lemma}

\newtheorem{corollary}[theorem]{Corollary}

\newenvironment{proof}
{\bigskip\noindent{\sc Proof.}\ \ \rm }{\hfill$\Box$\bigskip}

\title{A H\"older--Young--Lieb inequality for norms of Gaussian
Wick products}

\author{Paolo Da Pelo$^1$, Alberto Lanconelli$^1$ \quad and \quad  Aurel I. Stan$^2$}

\date{\empty}

\begin{document}

\maketitle
\begin{center}
{\noindent
\begin{tabular}{cc}
& $^1$Dipartimento di Matematica\\
& Universita' degli Studi di Bari\\
& Via E. Orabona, 4\\
& 70125 Bari - Italia\\
& E-mail: \emph{dapelo@dm.uniba.it}\\
\end{tabular}}
{\noindent
\begin{tabular}{cc}
& $^1$Dipartimento di Matematica\\
& Universita' degli Studi di Bari\\
& Via E. Orabona, 4\\
& 70125 Bari - Italia\\
& E-mail: \emph{lanconelli@dm.uniba.it}\\
\end{tabular}}
{\noindent
\begin{tabular}{cc}
& $^2$Department of Mathematics\\
& Ohio State University at Marion\\
& 1465 Mount Vernon Avenue\\
& Marion, OH 43302, U.S.A.\\
& E-mail: {\em stan.7@osu.edu}\\
\end{tabular}
}
\end{center}

\numberwithin{equation}{section}

\bigskip

\begin{abstract} An important connection between the finite dimensional
Gaussian Wick product and Lebesgue convolution product will be
proven first. Then  this connection will be used to prove an
important H\"older inequality for the norms of Gaussian Wick
products, reprove Nelson hypercontractivity inequality, and prove a
more general inequality whose marginal cases are the H\"older and
Nelson inequalities mentioned before. We will show that there is a
deep connection between the Gaussian H\"older inequality and classic
H\"older inequality, between the Nelson hypercontractivity and
classic Young inequality with the sharp constant, and between the
third more general inequality and an extension by Lieb of the Young
inequality with the best constant. Since the Gaussian probability
measure exists even in the infinite dimensional case, the above
three inequalities can be extended, via a classic Fatou's lemma
argument, to the infinite dimensional framework.
\end{abstract}



\section{Introduction}

If $(\Omega$, ${\mathcal F}$, $P)$ is a probability space and $H$ is
a closed subspace of $L^2(\Omega$, ${\mathcal F}$, $P)$, such that
every element $h$ from $H$ is normally distributed, with mean zero,
then $H$ is called a ({\em centered}) {\em Gaussian Hilbert space}.
If $H$ is a separable Gaussian Hilbert space and ${\mathcal F}(H)$
denotes the smallest sub--sigma algebra of ${\mathcal F}$, with
respect to which all random variables $h$ from $H$ are measurable,
then it was proven in \cite{kss2002} that for any two
complex--valued functions $\varphi$ and $\psi$ in $L^2(\Omega$,
${\mathcal F}(H)$, $P)$, and any two positive numbers $u$ and $v$,
such that $(1/u) + (1/v) = 1$, the Wick product of
$\Gamma((1/\sqrt{u})I)\varphi$ and $\Gamma((1/\sqrt{v}I))\psi$,
denoted by $\Gamma((1/\sqrt{u})I)\varphi \diamond
\Gamma((1/\sqrt{v})I)\psi$, belongs to $L^2(\Omega$, ${\mathcal
F}(H)$, $P)$ and the following inequality holds:
\begin{eqnarray}
\left\|\Gamma\left(\frac{1}{\sqrt{u}}I\right)\varphi \diamond
\Gamma\left(\frac{1}{\sqrt{v}}I\right)\psi\right\|_2 & \leq &
\|\varphi\|_2 \cdot \|\psi\|_2, \label{kuosaitostan}
\end{eqnarray}
where $\Gamma(cI)$ denotes the second quantization operator of $c$
times the identity operator $I$, for any complex constant $c$, and
$\| \cdot \|_2$ the $L^2$--norm. The proof was based on the
orthogonal structure of the space $L^2(\Omega$, ${\mathcal F}(H)$,
$P)$ (Fock decomposition) and Cauchy--Buniakovski--Schwarz
inequality. The authors of \cite{kss2002} regarded inequality
(\ref{kuosaitostan}) as a Young inequality for White Noise Analysis,
thinking that the Wick product, in the Gaussian case, is an analogue
of the convolution product, from the classic Fourier Analysis.
However, after discussing with other mathematicians, they were
convinced that this inequality should be called a H\"older
inequality for White Noise Analysis, since the Wick product is an
analogue of the classic product of two series.
\par The Wick product can be defined not only in
the Gaussian case, but also for any probability measure $\mu$ on
${\mathbb R}$ having finite moments of all orders. In \cite{ls2008}
it was proven that for any probability measure $\mu$, that is not a
delta measure (that means whose support does not reduce to a single
point), if $r \geq 2$ is a fixed number, and $u$ and $v$ are
positive numbers, such that, the inequality
\begin{eqnarray}
\left\|\Gamma\left(\frac{1}{\sqrt{u}}I\right)\varphi \diamond
\Gamma\left(\frac{1}{\sqrt{v}}I\right)\psi\right\|_r & \leq &
\|\varphi\|_r \cdot \|\psi\|_r, \label{rkss}
\end{eqnarray}
holds for any $\varphi$ and $\psi$ in $L^r({\mathbb R}$, $\mu$), we
must have:
\begin{eqnarray}
\frac{1}{u} + \frac{1}{v} & \leq & 1. \label{minimal}
\end{eqnarray}
Therefore, if one can prove inequality (\ref{kuosaitostan}), then it
is the best inequality that he (she) can get not only in the
Gaussian case, but also for all non--trivial probability measures.
\par In \cite{ls2010} it was proven, in the Gaussian case, that if
$u$ and $v$ are positive numbers, such that $(1/u) + (1/v) = 1$,
then any two real valued functions $\varphi$ and $\psi$ in
$L^1(\Omega$, ${\mathcal F}(H)$, $P)$, the Wick product of
$\Gamma((1/\sqrt{u})I)\varphi \diamond \Gamma((1/\sqrt{v})I)\psi$
belongs to $L^1(\Omega$, ${\mathcal F}(H)$, $P)$, and the following
inequality holds:
\begin{eqnarray}
\left\|\Gamma\left(\frac{1}{\sqrt{u}}I\right)\varphi \diamond
\Gamma\left(\frac{1}{\sqrt{v}}I\right)\psi\right\|_1 & \leq &
\|\varphi\|_1 \cdot \|\psi\|_1. \label{lanconellistan1}
\end{eqnarray}
In \cite{ls2010} it was also proven, in the Gaussian case, that for
any two real valued functions $\varphi$ and $\psi$ in
$L^{\infty}(\Omega$, ${\mathcal F}(H)$, $P)$, and any two positive
numbers $u$ and $v$, such that $(1/u) + (1/v) = 1$, the Wick product
of $\Gamma((1/\sqrt{u})I)\varphi \diamond \Gamma((1/\sqrt{v})I)\psi$
belongs to $L^{\infty}(\Omega$, ${\mathcal F}(H)$, $P)$, and the
following inequality holds:
\begin{eqnarray}
\left\|\Gamma\left(\frac{1}{\sqrt{u}}I\right)\varphi \diamond
\Gamma\left(\frac{1}{\sqrt{v}}I\right)\psi\right\|_{\infty} & \leq &
\|\varphi\|_{\infty} \cdot \|\psi\|_{\infty}.
\label{lanconellistan2}
\end{eqnarray}
To prove the inequalities (\ref{lanconellistan1}) and
(\ref{lanconellistan2}), the authors proved first a Jensen
inequality for Gaussian Wick products, inspired by the Jensen
inequality from \cite{jan97}.
\par In this paper, we will prove first a lemma that connects the
Gaussian Wick product to the classic convolution product of the
Lebesgue measure. We will then use this lemma to prove the following
inequalities.
\par Let $p \in [1$, $\infty]$. If $u$ and $v$ are positive numbers,
such that $(1/u) + (1/v) = 1$, then for any two complex valued
functions $\varphi$ and $\psi$ in $L^p(\Omega$, ${\mathcal F}(H)$,
$P)$, the Wick product $\Gamma((1/\sqrt{u})I)\varphi \diamond
\Gamma((1/\sqrt{v})I)\psi$ belongs to $L^p(\Omega$, ${\mathcal
F}(H)$, $P)$, and the following inequality holds:
\begin{eqnarray}
\left\|\Gamma\left(\frac{1}{\sqrt{u}}I\right)\varphi \diamond
\Gamma\left(\frac{1}{\sqrt{v}}I\right)\psi\right\|_p & \leq &
\|\varphi\|_p \cdot \|\psi\|_p. \label{dls1}
\end{eqnarray}
We will show that via the lemma connecting the Gaussian Wick product
to the Lebesgue convolution product, inequality (\ref{dls1}) reduces
to the classic H\"older inequality.
\par Let $1 < p \leq r < \infty$. Then
for any $\varphi$ in $L^p(\Omega$, ${\mathcal F}(H)$ $P)$ and any
$\psi$ in $L^{\infty}(\Omega$, ${\mathcal F}(H)$ $P)$,
$\Gamma(\sqrt{p - 1}/\sqrt{r - 1})\varphi \diamond \Gamma(\sqrt{r -
p}/\sqrt{r - 1})\psi$ belongs to $L^{r}(\Omega$, ${\mathcal F}(H)$
$P)$ and the following inequality holds:
\begin{eqnarray}
\left\|\left[\Gamma\left(\frac{\sqrt{p - 1}}{\sqrt{r -
1}}\right)\varphi\right] \diamond \left[\Gamma\left(\frac{\sqrt{r -
p}}{\sqrt{r - 1}}\right)\psi\right]\right\|_r & \leq & \|\varphi\|_p
\cdot \|\psi\|_{\infty}. \label{dls2}
\end{eqnarray}
In particular, if we choose $\psi = 1$ (the constant random variable
equal to $1$), the we get the classic Nelson hypercontractivity
inequality:
\begin{eqnarray}
\left\|\Gamma\left(\frac{\sqrt{p - 1}}{\sqrt{r -
1}}\right)\varphi\right\|_r & \leq & \|\varphi\|_p.
\end{eqnarray}
We will show that via the lemma connecting the Gaussian Wick product
to the Lebesgue convolution product, this inequality reduces to the
Young inequality with the best constant proven by Beckner and
Brascamp--Lieb in \cite{b1975} and \cite{bl1976}, respectively.
\par Finally we will prove that if $u$ and $v$ are positive numbers,
such that $(1/u) + (1/v) = 1$, and $p$, $q$, and $r$ are in $[1$,
$\infty]$, such that:
\begin{eqnarray}
\frac{1}{r - 1} & = & \frac{1}{u(p - 1)} + \frac{1}{v(q - 1)},
\label{condition}
\end{eqnarray}
then for any $\varphi$ in $L^p(\Omega$, ${\mathcal F}(H)$, $P)$ and
any $\psi$ in $L^q(\Omega$, ${\mathcal F}(H)$, $P)$,
$\Gamma((1/\sqrt{u})I)\varphi \diamond \Gamma((1/\sqrt{v})I)\psi$
belongs to $L^r(\Omega$, ${\mathcal F}(H)$, $P)$, and the following
inequality holds:
\begin{eqnarray}
\left\|\Gamma\left(\frac{1}{\sqrt{u}}I\right)\varphi \diamond
\Gamma\left(\frac{1}{\sqrt{v}}I\right)\psi\right\|_r & \leq &
\|\varphi\|_p \cdot \|\psi\|_q. \label{dls3}
\end{eqnarray}
We will show that via the lemma connecting the Gaussian Wick product
to the Lebesgue convolution product, this inequality is connected to
the fully generalized Young's inequality proven by Lieb in
\cite{l1990}. See also page 100 of \cite{ll2001}.
\par It is easy to see that (\ref{dls1}) and (\ref{dls2}) are
particular cases of (\ref{dls3}), namely, (\ref{dls3}) reduces to
(\ref{dls1}) in the particular case $p = q = r$, and (\ref{dls3})
reduces to (\ref{dls2}) when $q = \infty$. In fact, condition
(\ref{condition}) tells us that $1/(r - 1)$ is a convex combination
of $1/(p - 1)$ and $1/(q - 1)$, and so, if we assume that $p \leq
q$, then $p \leq r \leq q$. If we fix the left--endpoint $p$, of the
interval $[p$, $q]$, and let the right--endpoint $q$ vary from $p$
to $\infty$, then inequalities (\ref{dls1}) and (\ref{dls2}) are the
``marginal" cases of (\ref{dls3}): $q = p$ and $q = \infty$,
respectively. So, one might say that proving (\ref{dls3}) makes the
proofs of (\ref{dls1}) and (\ref{dls2}) superfluous. However, we
prefer to prove first (\ref{dls1}), then (\ref{dls2}), and finally
(\ref{dls3}), to show how they are connected to the following
important and deep inequalities of classical Analysis: H\"older,
Young with the sharp constant, and Lieb. We will also see that as we
move from inequality (\ref{dls1}) to (\ref{dls3}), the complexity of
the proof increases more and more.
\par All the above inequalities are sharp, the equality occurring for
some exponential functions.
\par In section 2, we present a short background of the theory of
Gaussian Hilbert spaces. This background includes the definition of
the Wick product and second quantization operator of a constant
times the identity. In section 3 we prove the important lemma
connecting the Gaussian Wick product to the Lebesgue convolution
product. Finally in the last section we prove the main results of
this paper.

\section{Background} In this section we present a minimal background
about Gaussian Wick products and second quantization operators. The
frameworks in which this background can be presented are many. One
can use, for example, Hida's White Noise Theory (see \cite{k1996} or
\cite{o1994}), Malliavin Calculus, the multiple Wiener integrals, a
Fock space, or the theory of Gaussian Hilbert Spaces. All of these
theories are leading to the same notion of Wick product. Since we
are not going to use generalized functions, we are going to use the
theory of Gaussian Hilbert Spaces as outlined in \cite{jan97}. \\
\par Let $(\Omega$, ${\mathcal F}$, $P)$ be a probability space and
$H$ a closed subspace of $L^2(\Omega$, ${\mathcal F}$, $P)$, such
that every element $h$ of $H$ is normally distributed with mean
zero. We call $H$ a ({\em centered}) {\em Gaussian Hilbert space}.
We assume that $H$ is separable. For all non--negative integers $n$,
we define the space:
\begin{eqnarray*}
F_n & = & \{f(h_1, \dots, h_d) \mid d \geq 1, h_i \in H, i = 1,
\dots, d, f \ {\rm is} \ {\rm polynomial}, \ deg(f) \leq n\},
\end{eqnarray*}
where $deg(f)$ denotes the degree of the polynomial $f$. Since the
Gaussian random variables have finite moments of all orders, each
space $F_n$ is contained in $L^2(\Omega$, ${\mathcal F}$, $P)$. We
have:
\[
{\mathbb C} = F_0 \subset F_1 \subset F_2 \subset \cdots \subset
L^{2}(\Omega, {\mathcal F}, P).
\]
We define now the following spaces: $G_0 := F_0$ and for all $n \geq
1$,
\begin{eqnarray*}
G_n & := & \bar{F}_n \ominus \bar{F}_{n - 1},
\end{eqnarray*}
where $\bar{F}$ denotes the closure of $F$ in $L^2(\Omega$,
${\mathcal F}$, $P)$, for any subspace $F$ of $L^2(\Omega$,
${\mathcal F}$, $P)$. For each non--negative integer $n$, we call
$G_n$ {\em the $n$--th homogenous chaos space generated by $H$}, and
every element $\varphi$ from $G_n$, a {\em homogenous polynomial
random variable of degree $n$}. We define the following Hilbert
space:
\begin{eqnarray*}
{\mathcal H} & = & \oplus_{n = 0}^{\infty}G_n
\end{eqnarray*}
and call it {\em the chaos space generated by $H$}. Let us observe
that every random variable from ${\mathcal H}$ is measurable with
respect to the sigma--algebra ${\mathcal F}(H)$ generated by the
elements $h$ from $H$. The reciprocal is also true and the following
theorem holds (see Theorem 2.6., from page 18, in \cite{jan97}).
\begin{theorem} \label{funjanson}
\begin{eqnarray}
{\mathcal H} & = & L^2(\Omega, {\mathcal F}(H), P).
\end{eqnarray}
\end{theorem}
From now on, because of Theorem \ref{funjanson}, whatever random
variables we will consider, they will be measurable with respect to
${\mathcal F}(H)$.\\
\par For every $n \geq 0$, we denote by $P_n$ the orthogonal projection
from ${\mathcal H}$ onto $G_n$. The Wick product is defined first
for any two homogenous polynomial random variables, and then
extended in a bilinear way, as explained below. For any
non--negative integers $m$ and $n$, and any $\varphi$ in $G_m$ and
$\psi$ in $G_n$, we define:
\begin{eqnarray}
\varphi \diamond \psi & = & P_{m + n}(\varphi \cdot \psi).
\label{wickdef1}
\end{eqnarray}
If $\varphi = \sum_{n = 0}^{\infty}f_n \in {\mathcal H}$ and $\psi =
\sum_{n = 0}^{\infty}g_n \in {\mathcal H}$, where $f_n$ and $g_n$
are in $G_n$, for all $n \geq 0$, then the {\em Wick product of
$\varphi$ and $\psi$}, denoted by $\varphi \diamond \psi$, is
defined as:
\begin{eqnarray}
\varphi \diamond \psi & = & \sum_{k = 0}^{\infty}\left[\sum_{p + q =
k}(f_p \diamond g_q)\right]. \label{wickdef2}
\end{eqnarray}
Of course, there might be problems with the convergence, in the
$L^2$--sense, of the series from the right--hand side of
(\ref{wickdef2}), but at least for the case when $\varphi$ and
$\psi$ are polynomial random variables (i.e., only finitely many
$f_n$ and $g_n$, $n \geq 1$, are different from zero), the Wick
product $\varphi \diamond \psi$ is well--defined.\\
\par If $c$ is a fixed complex number, then we define {\em the second
quantization operator of} $cI$, where $I$ denotes the identity
operator of $H$, by:
\begin{eqnarray}
\Gamma(cI)\varphi & = & \sum_{n = 0}^{\infty}c^nf_n,
\end{eqnarray}
for all $\varphi = \sum_{n = 0}^{\infty}f_n \in {\mathcal H}$, where
$f_n$ is in $G_n$, for all $n \geq 0$. From now we are going to drop
the letter $I$, and write simply $\Gamma(c)$ instead of
$\Gamma(cI)$. It is clear, that if $|c| \leq 1$, we have
$\Gamma(c)\varphi \in {\mathcal H}$, and the following inequality
holds:
\begin{eqnarray}
\|\Gamma(c)\varphi\|_2 & \leq & \|\varphi\|_2,
\end{eqnarray}
where $\| \cdot \|_2$ denotes the $L^2$--norm. Moreover, as it is
shown in \cite{jan97}, Theorem 4.12, page 48, if $c$ is real and
$|c| \leq 1$, then the second quantization operator $\Gamma(c)$ has
a unique continuous extension from $L^1(\Omega$, ${\mathcal F}(H)$,
$P)$ to $L^1(\Omega$, ${\mathcal F}(H)$, $P)$, that we denote also
by $\Gamma(c)$, and this extension is a bounded linear operator, of
operatorial norm equal to $1$, from $L^p(\Omega$, ${\mathcal F}(H)$,
$P)$ to $L^p(\Omega$, ${\mathcal F}(H)$, $P)$, for all $1 \leq p
\leq
\infty$.\\
\par The second quantization operators are distributive with respect
to the Wick product, in the following sense:
\begin{lemma}
For all $c$ in ${\mathbb C}$, such that $|c| \leq 1$, and all
$\varphi$ and $\psi$ in $L^2(\Omega$, ${\mathcal F}(H)$, $P)$, such
that $\varphi \diamond \xi$ belongs to $L^2(\Omega$, ${\mathcal
F}(H)$, $P)$, we have:
\begin{eqnarray}
\Gamma(c)(\varphi \diamond \psi) & = & \Gamma(c)\varphi \diamond
\Gamma(c)\psi.
\end{eqnarray}
\end{lemma}
The second quantization operators are very important in assuring the
convergence of the Wick product of two random variables, due to the
following theorems, from \cite{kss2002} and \cite{ls2010},
respectively.
\begin{theorem}
For any $u$ and $v$ positive numbers, such that $(1/u) + (1/v) = 1$,
and for any $\varphi$ and $\psi$ in $L^2(\Omega$, ${\mathcal F}(H)$,
$P)$, the Wick product $\Gamma(1/\sqrt{u})\varphi \diamond
\Gamma(1/\sqrt{v})\psi$ belongs to $L^2(\Omega$, ${\mathcal F}(H)$,
$P)$, and the following inequality holds:
\begin{eqnarray}
\left\|\Gamma\left(\frac{1}{\sqrt{u}}\right)\varphi \diamond
\Gamma\left(\frac{1}{\sqrt{v}}\right)\psi\right\|_2 & \leq &
\|\varphi\|_2 \cdot \|\psi\|_2.
\end{eqnarray}
\end{theorem}
\begin{theorem}
For any $u$ and $v$ positive numbers, such that $(1/u) + (1/v) = 1$,
the bilinear operator $T_{u, v} : L^2(\Omega$, ${\mathcal F}(H)$,
$P) \times L^2(\Omega$, ${\mathcal F}(H)$, $P) \to L^2(\Omega$,
${\mathcal F}(H)$, $P)$, defined as:
\begin{eqnarray}
T_{u, v}(\varphi, \psi) & = &
\Gamma\left(\frac{1}{\sqrt{u}}\right)\varphi \diamond
\Gamma\left(\frac{1}{\sqrt{v}}\right)\psi,
\end{eqnarray}
for any $\varphi$ and $\psi$ in $L^2(\Omega$, ${\mathcal F}(H)$,
$P)$, admits a unique continuous linear extension $\tilde{T}_{u, v}$
from $L^1(\Omega$, ${\mathcal F}(H)$, $P) \times L^1(\Omega$,
${\mathcal F}(H)$, $P)$ to $L^1(\Omega$, ${\mathcal F}(H)$, $P)$.
Moreover, if $\varphi$ and $\psi$ are real valued, then we have:
\begin{eqnarray}
\left\|\Gamma\left(\frac{1}{\sqrt{u}}\right)\varphi \diamond
\Gamma\left(\frac{1}{\sqrt{v}}\right)\psi\right\|_1 & \leq &
\|\varphi\|_1 \cdot \|\psi\|_1
\end{eqnarray}
and
\begin{eqnarray}
\left\|\Gamma\left(\frac{1}{\sqrt{u}}\right)\varphi \diamond
\Gamma\left(\frac{1}{\sqrt{v}}\right)\psi\right\|_{\infty} & \leq &
\|\varphi\|_{\infty} \cdot \|\psi\|_{\infty}.
\end{eqnarray}
\end{theorem}
In this paper, we will remove the condition ``real valued" from the
last sentence of the previous theorem.\\
\par There is an important family of random variables in this
theory, that have some beautiful properties with respect to the Wick
product and second quantization operators. These functions are
called the ({\em renormalized}){\em exponential random variables},
and are defined as follows. For any $\xi \in H_c$ (where $H_c$
denotes the complexification of $H$), we define the {\em exponential
function} ${\mathcal E}_{\xi}$ {\em generated by} $\xi$, by the
formula:
\begin{eqnarray}
{\mathcal E}_{\xi} & := & \sum_{n = 0}^{\infty}\frac{1}{n!}
\xi^{\diamond n}, \label{expdef1}
\end{eqnarray}
where $\xi^{\diamond n} : = \xi \diamond \xi \diamond \cdots
\diamond \xi$ ($n$ times). As a random variable, ${\mathcal
E}_{\xi}$ can be written as:
\begin{eqnarray}
{\mathcal E}_{\xi}(\omega) & = & e^{\xi(\omega) - \frac{1}{2}\langle
\xi, \xi \rangle}, \label{expdef2}
\end{eqnarray}
for all $\omega \in \Omega$, where:
\begin{eqnarray}
\langle f, g \rangle & := & E[f \cdot g],
\end{eqnarray}
for all $f$ and $g$ in $H_c$, and $E$ denotes the expectation. It
can be easily seen that ${\mathcal E}_{\xi}$ belongs to
$L^p(\Omega$, ${\mathcal F}(H)$, $P)$, for all $1 \leq p < \infty$,
and all $\xi$ in $H_c$. The family of exponential functions is
closed with respect to the Wick product and second quantization
operators, as illustrated by the following lemma (see \cite{k1996}
and \cite{o1994}).
\begin{lemma}
For all $\xi$ and $\eta$ in $H_c$, and all $c$ in ${\mathbb C}$, we
have:
\begin{eqnarray}
{\mathcal E}_{\xi} \diamond {\mathcal E}_{\eta} & = & {\mathcal
E}_{\xi + \eta}
\end{eqnarray}
and
\begin{eqnarray}
\Gamma(c){\mathcal E}_{\xi} & = & {\mathcal E}_{c\xi}.
\end{eqnarray}
\end{lemma}
Finally, the exponential functions are important in defining the
$S$--transform. If $\varphi \in L^2(\Omega$, ${\mathcal F}(H)$, $P)$
and $\xi \in H_c$, then the {\em $S$--transform} of $\varphi$ at
$\xi$ is defined by the formula:
\begin{eqnarray}
(S\varphi)(\xi) & = & E\left[\varphi \cdot {\mathcal
E}_{\xi}\right].
\end{eqnarray}
The $S$--transform is a unitary operator from $L^2(\Omega$,
${\mathcal F}(H)$, $P)$ onto a Hilbert space of holomorphic
functions, which we are not going to describe, since it is not
important in this paper. If one considers the $S$--transform as an
analogue of Fourier or Laplace transforms from the classic Analysis,
then the Wick product becomes automatically an analogue of the
convolution product, due to the following easy to check property.
\begin{lemma}
For any two functions $\varphi$ and $\psi$ in $L^2(\Omega$,
${\mathcal F}(H)$, $P)$, such that $\varphi \diamond \psi$ belongs
to $L^2(\Omega$, ${\mathcal F}(H)$, $P)$, we have:
\begin{eqnarray}
S(\varphi \diamond \psi) & = & (S\varphi) \cdot (S\psi).
\end{eqnarray}
\end{lemma}

\section{An connection between the Gaussian Wick product
and Lebesgue convolution product}

 Let $d$ be a fixed positive integer. Let $d_Nx$ denote the
normalized Lebesgue measure on ${\mathbb R}^d$,
$(1/\sqrt{2\pi})^ddx$, and $\mu$ the standard Gaussian probability
measure on ${\mathbb R}^d$, i.e., $d\mu = e^{-\langle x, x
\rangle/2}d_Nx$. If $X_i : \Omega \to {\mathbb R}$, $i = \{1$, $2$,
$\dots$, $d\}$ are independent standard normal random variables, and
${\mathcal F}$ is the sigma--algebra generated by them, then any
random variable $Y : \Omega \to {\mathbb C}$, that is measurable
with respect ${\mathcal F}$, can be written as $Y = g(X_1$, $X_2$,
$\dots$, $X_d)$, where $g : {\mathbb R}^d \to {\mathbb C}$ is a
Borel measurable function. From now on we will write $g(x)$ with a
lower case $x$ instead of the upper case $X$, and do the
computations of integrals in terms of the probability distribution
$\mu$ of $X$, where $X := (X_1$, $X_2$, $\dots$, $X_d)$. Observe
that in terms of distributions, for any $p \geq 1$, the $L^p$ norm
of a function $f(x)$ with respect to the Gaussian measure $\mu$ is
the same as the $L^p$ norm of $f(x)e^{-\langle x, x \rangle/(2p)}$
with respect to the normalized Lebesgue measure $d_Nx$. This simple
fact will be used throughout this paper. Everything will be done
using the normalized Lebesgue measure. Throughout this paper, for
any $p \in [1$, $\infty]$, we will denote by $\| \cdot \|_p$ and
$\|| \cdot |\|_p$ the $L^p$ norms with respect to the Gaussian
measure $\mu$ and normalized Lebesgue measure $d_Nx$, respectively.
\par We are now presenting a connection between the Gaussian Wick product
$\diamond$ and convolution product $\star$ with respect to the
normalized Lebesgue measure.

\begin{lemma}\label{convolution}
Let $u$ and $v$ be positive numbers, such that: $(1/u) + (1/v) = 1$.
Then for any $\varphi$ and $\psi$ in $L^1({\mathbb R}^d$, $\mu)$, we
have that $\Gamma(1/\sqrt{u})\varphi \diamond
\Gamma(1/\sqrt{v})\psi$ belongs to $L^1({\mathbb R}^d$, $\mu)$ and:
\begin{eqnarray}
& \ & \left[\varphi
\left(\frac{1}{\sqrt{v}}x\right)e^{-\frac{\langle x, x
\rangle}{2v}}\right] \star \left[\psi
\left(\frac{1}{\sqrt{u}}x\right)e^{-\frac{\langle x, x
\rangle}{2u}}\right]
\nonumber\\
& = & \left[\Gamma\left(\frac{1}{\sqrt{u}}\right)\varphi \diamond
\Gamma\left(\frac{1}{\sqrt{v}}\right)\psi\right]
\left(\frac{1}{\sqrt{uv}}x\right)e^{-\frac{1}{2uv}\langle x, x
\rangle}, \label{convWick}
\end{eqnarray}
\end{lemma}
where the convolution product in the left--hand side is computed
with respect to the normalized Lebesgue measure $d_Nx$.

We prove first an easier version of this lemma.

\begin{lemma}\label{connection}
Let ${\mathcal E} = \{\sum_{i = 1}^nc_ie^{\langle \xi_i, x \rangle -
(1/2)\langle \xi_i, \xi_i \rangle} \mid n \in {\mathbb N}, c_i \in
{\mathbb C}, \xi_i \in {\mathbb C}^d, \forall i \in \{1, 2, \dots,
n\}\}$. Let $u$ and $v$ be positive numbers, such that $(1/u) +
(1/v) = 1$. For any $\varphi$ and $\psi$ in ${\mathcal E}$, we have:
\begin{eqnarray}
& \ & \left\{\left[\Gamma(\sqrt{u})\varphi\right]
\left(\frac{x}{\sqrt{v}}\right)e^{-\frac{\langle x, x
\rangle}{2v}}\right\} \star \left\{\left[\Gamma(\sqrt{v})\psi\right]
\left(\frac{x}{\sqrt{u}}\right)e^{-\frac{\langle x, x
\rangle}{2u}}\right\}
\nonumber\\
& = & (\varphi \diamond
\psi)\left(\frac{x}{\sqrt{uv}}\right)e^{-\frac{\langle x, x
\rangle}{2uv}}. \label{fundamental}
\end{eqnarray}
In particular, replacing $\varphi$ and $\psi$ by
$\Gamma(1/\sqrt{u})\varphi$ and $\Gamma(1/\sqrt{v})\psi$,
respectively, we obtain that (\ref{convWick}) holds for any two
functions that are linear combinations of exponential functions.
\end{lemma}
\begin{proof}
Since both sides of (\ref{fundamental}) are bilinear with respect to
$\varphi$ and $\psi$, it is enough to check the relation for
$\varphi(x) = e^{\langle \xi, x \rangle - (1/2)\langle \xi, \xi
\rangle}$ and $\psi(x) = e^{\langle \eta, x \rangle - (1/2)\langle
\eta, \eta \rangle}$, where $\xi$ and $\eta$ are arbitrarily fixed
vectors in ${\mathbb C}^d$. Indeed, for these functions we have:
\begin{eqnarray*}
& \ & \left\{\left[\Gamma(\sqrt{u})\varphi\right]
\left(\frac{\cdot}{\sqrt{v}}\right)e^{-\frac{\langle \cdot, \cdot
\rangle}{2v}}\right\} \star \left\{\left[\Gamma(\sqrt{v})\psi\right]
\left(\frac{\cdot}{\sqrt{u}}\right)e^{-\frac{\langle \cdot,
\cdot \rangle}{2u}}\right\}(x)\\
& = & \left[e^{\frac{\sqrt{u}}{\sqrt{v}}\langle \xi, \cdot \rangle -
\frac{u}{2}\langle \xi, \xi \rangle - \frac{\langle \cdot, \cdot
\rangle}{2v}}\right] \star \left[e^{\frac{\sqrt{v}}{\sqrt{u}}\langle
\eta, \cdot \rangle - \frac{v}{2}\langle \eta, \eta \rangle -
\frac{\langle \cdot, \cdot \rangle}{2u}}\right](x)\\
& = & \int_{{\mathbb R}^d}e^{\frac{\sqrt{u}}{\sqrt{v}}\langle \xi, x
- y \rangle - \frac{u}{2}\langle \xi, \xi \rangle - \frac{\langle x
- y, x - y \rangle}{2v}} \cdot e^{\frac{\sqrt{v}}{\sqrt{u}}\langle
\eta, y \rangle - \frac{v}{2}\langle \eta, \eta \rangle -
\frac{\langle y,
y \rangle}{2u}}d_Ny\\
& = & e^{\frac{\sqrt{u}}{\sqrt{v}}\langle \xi, x \rangle -
\frac{u}{2}\langle \xi, \xi \rangle - \frac{v}{2}\langle \eta, \eta
\rangle - \frac{1}{2v}\langle x, x \rangle}\int_{{\mathbb
R}^d}e^{-\frac{1}{2}\langle y, y \rangle + \left\langle
-\frac{\sqrt{u}}{\sqrt{v}}\xi + \frac{\sqrt{v}}{\sqrt{u}}\eta +
\frac{1}{v}x, y \right\rangle}d_Ny.
\end{eqnarray*}
We now perform the classic trick of completing the square, in the
exponential, by subtracting and adding $(1/2)\langle
-(\sqrt{u}/\sqrt{v})\xi + (\sqrt{v}/\sqrt{u})\eta + (1/v)x,
-(\sqrt{u}/\sqrt{v})\xi + (\sqrt{v}/\sqrt{u})\eta + (1/v)x \rangle$,
a factor that does not depend on the variable of integration $y$,
and can be taken out of the integral. Thus, we obtain:
\begin{eqnarray*}
& \ & \left\{\left[\Gamma(\sqrt{u})\varphi\right]
\left(\frac{\cdot}{\sqrt{v}}\right)e^{-\frac{\langle \cdot, \cdot
\rangle}{2v}}\right\} \star \left\{\left[\Gamma(\sqrt{v})\psi\right]
\left(\frac{\cdot}{\sqrt{u}}\right)e^{-\frac{\langle \cdot,
\cdot \rangle}{2u}}\right\}(x)\\
& = & e^{\frac{\sqrt{u}}{\sqrt{v}}\langle \xi, x \rangle -
\frac{u}{2}\langle \xi, \xi \rangle - \frac{v}{2}\langle \eta, \eta
\rangle - \frac{1}{2v}\langle x, x \rangle} \cdot
e^{\frac{1}{2}\left\langle -\frac{\sqrt{u}}{\sqrt{v}}\xi +
\frac{\sqrt{v}}{\sqrt{u}}\eta + \frac{1}{v}x,
-\frac{\sqrt{u}}{\sqrt{v}}\xi + \frac{\sqrt{v}}{\sqrt{u}}\eta +
\frac{1}{v}x \right\rangle}\\
& \ & \times \int_{{\mathbb R}^d}e^{-\frac{1}{2}\left\langle y -
\frac{\sqrt{u}}{\sqrt{v}}\xi + \frac{\sqrt{v}}{\sqrt{u}}\eta +
\frac{1}{v}x , y - \frac{\sqrt{u}}{\sqrt{v}}\xi +
\frac{\sqrt{v}}{\sqrt{u}}\eta + \frac{1}{v}x \right\rangle}d_Ny\\
& = & e^{-\frac{u}{2}\left(1 - \frac{1}{v}\right)\langle \xi, \xi
\rangle - \langle \xi, \eta \rangle - \frac{v}{2}\left(1 -
\frac{1}{u}\right)\langle \eta, \eta \rangle +
\frac{\sqrt{u}}{\sqrt{v}}\left(1 - \frac{1}{v}\right)\langle \xi, x
\rangle + \frac{1}{\sqrt{uv}}\langle \eta, x \rangle -
\frac{1}{2v}\left(1 - \frac{1}{v}\right)\langle x, x \rangle}
\int_{{\mathbb R}^d}e^{-\frac{1}{2}\langle z, z \rangle}d_Nz\\
& = & e^{-\frac{u}{2}\cdot \frac{1}{u}\langle \xi, \xi \rangle -
\langle \xi, \eta \rangle - \frac{v}{2} \cdot \frac{1}{v}\langle
\eta, \eta \rangle + \frac{\sqrt{u}}{\sqrt{v}} \cdot
\frac{1}{u}\langle \xi, x \rangle + \frac{1}{\sqrt{uv}} \cdot
\langle \eta, x \rangle -
\frac{1}{2v} \cdot \frac{1}{u}\langle x, x \rangle} \cdot 1\\
& = & e^{-\frac{1}{2}\langle \xi + \eta, \xi + \eta \rangle +
\langle \xi + \eta, \frac{1}{\sqrt{uv}}x \rangle - \frac{1}{2uv}
\cdot \langle x, x \rangle}\\
& = & (\varphi \diamond
\psi)\left(\frac{x}{\sqrt{uv}}\right)e^{-\frac{\langle x, x
\rangle}{2uv}},
\end{eqnarray*}
since $(\varphi \diamond \psi)(x) = e^{\langle \xi + \eta, x \rangle
- (1/2)\langle \xi + \eta, \xi + \eta \rangle}$.
\end{proof}

A complete proof of Lemma \ref{convolution} is the following:

\begin{proof}
Since ${\mathcal E}$ is dense in $L^1({\mathbb R}^d$, $\mu)$, there
exist two sequences $\{f_n\}_{n \geq 1}$ and $\{g_n\}_{n \geq 1}$ of
elements of ${\mathcal E}$, such that: $\parallel f_n - \varphi
\parallel_1 \to 0$ and $\parallel g_n - \psi
\parallel_1 \to 0$, as $n \to \infty$. For each $n \geq 1$, we have:
\begin{eqnarray*}
& \ & \left[f_n \left(\frac{x}{\sqrt{v}}\right)e^{-\frac{\langle x,
x \rangle }{2v}}\right] \star \left[g_n
\left(\frac{x}{\sqrt{u}}\right)e^{-\frac{\langle x, x
\rangle}{2u}}\right]
\nonumber\\
& = & \left[\Gamma\left(\frac{1}{\sqrt{u}}\right)f_n \diamond
\Gamma\left(\frac{1}{\sqrt{v}}\right)g_n\right]
\left(\frac{x}{\sqrt{uv}}\right)e^{-\frac{\langle x, x
\rangle}{2uv}}.
\end{eqnarray*}
The left--hand side converges to $\displaystyle \left[\varphi
\left(\frac{x}{\sqrt{v}}\right)e^{-\frac{\langle x, x
\rangle}{2v}}\right] \star \left[\psi
\left(\frac{x}{\sqrt{u}}\right)e^{-\frac{\langle x, x
\rangle}{2u}}\right]$, in $L^1({\mathbb R}^d$, $\mu)$, as $n \to
\infty$, by Young inequality for the normalized Lebesgue measure.
The right--hand side converges to
$\left[\Gamma\left(\frac{1}{\sqrt{u}}\right)\varphi \diamond
\Gamma\left(\frac{1}{\sqrt{v}}\right)\psi\right]\left(\frac{x}{\sqrt{uv}}\right)
e^{-\frac{\langle x, x \rangle}{2uv}}$, by Lanconelli--Stan
inequality, from \cite{ls2010}, about the $L^1$ norms for Wick
products.
\end{proof}

\section{Main results}

\begin{theorem} {\bf (Finite dimensional H\"older inequality for Gaussian Wick
products.)} \ Let $d \in {\mathbb N}$ and $p \in [1$, $\infty]$ be
fixed. Let $\mu$ denote the standard Gaussian probability measure on
${\mathbb R}^d$. Let $u$ and $v$ be positive numbers, such that
$(1/u) + (1/v) = 1$. Then for any $\varphi$ and $\psi$ in
$L^p({\mathbb R}^d$, $\mu)$, $\Gamma(1/\sqrt{u})\varphi \diamond
\Gamma(1/\sqrt{v})\psi \in L^p({\mathbb R}^d$, $\mu)$ and the
following inequality holds:
\begin{eqnarray}
\left\|\Gamma\left(\frac{1}{\sqrt{u}}\right)\varphi \diamond
\Gamma\left(\frac{1}{\sqrt{v}}\right)\psi\right\|_p & \leq &
\|\varphi \|_p \cdot \| \psi \|_p. \label{Holder}
\end{eqnarray}
\end{theorem}

\begin{proof} Let $p'$ be the conjugate of $p$, i.e., $(1/p) + (1/p') = 1$.
Let's multiply both sides of formula (\ref{convWick}) by $e^{\langle
x, x \rangle/(2p'uv)}$. We get:
\begin{eqnarray*}
& \ & \left[\Gamma\left(\frac{1}{\sqrt{u}}\right)\varphi \diamond
\Gamma\left(\frac{1}{\sqrt{v}}\right)\psi\right]
\left(\frac{x}{\sqrt{uv}}\right)e^{-\frac{\langle
x, x \rangle}{2uv}\left(1 - \frac{1}{p'}\right)}\\
& = & e^{\frac{\langle x, x \rangle}{2p'uv}} \left\{\left[\varphi
\left(\frac{x}{\sqrt{v}}\right)e^{-\frac{\langle x, x
\rangle}{2v}}\right] \star
\left[\psi\left(\frac{x}{\sqrt{u}}\right)e^{-\frac{\langle x, x
\rangle}{2u}}\right]\right\}\\
& = & \int_{{\mathbb R}^d}\left[\varphi\left(\frac{x -
y}{\sqrt{v}}\right)e^{-\frac{\langle x - y, x - y
\rangle}{2pv}}\right]e^{-\frac{\langle x - y, x - y \rangle}{2p'v}}
\left[\psi\left(\frac{y}{\sqrt{u}}\right)e^{-\frac{\langle y, y
\rangle}{2pu}}\right]e^{-\frac{\langle y, y \rangle}{2p'u}} \cdot
e^{\frac{\langle x, x \rangle}{2p'uv}}d_Ny.
\end{eqnarray*}
Let $f(x) := \varphi((1/\sqrt{v})x)\exp(-[1/(2pv)]\langle x, x
\rangle)$ and $g(x) := \psi((1/\sqrt{u})x)e(-[1/(2pu)]\langle x, x
\rangle)$. With these notations we get:
\begin{eqnarray*}
& \ & \left[\Gamma\left(\frac{1}{\sqrt{u}}\right)\varphi \diamond
\Gamma\left(\frac{1}{\sqrt{v}}\right)\psi\right]
\left(\frac{x}{\sqrt{uv}}\right)e^{-\frac{\langle
x, x \rangle}{2uv} \cdot \frac{1}{p}}\\
& = & \int_{{\mathbb R}^d}f(x - y)g(y)e^{-\frac{1}{2p'}\left\langle
y - \frac{1}{v}x, y - \frac{1}{v}x\right\rangle}d_Ny.
\end{eqnarray*}
Putting the modulus in both sides, then introducing it in the
integral in the right (triangle inequality), and then applying the
H\"older inequality to the pair $(p$, $p')$, we get:
\begin{eqnarray*}
& \ & \left|\left[\Gamma\left(\frac{1}{\sqrt{u}}\right)\varphi
\diamond \Gamma\left(\frac{1}{\sqrt{v}}\right)\psi\right]
\left(\frac{x}{\sqrt{uv}}\right)e^{-\frac{\langle
x, x \rangle}{2puv}}\right|\\
& \leq & \left[\int_{{\mathbb R}^d}|f(x - y)g(y)|^p
d_Ny\right]^{\frac{1}{p}} \left[\int_{{\mathbb
R}^d}\left(e^{-\frac{1}{2p'}\left\langle y - \frac{1}{v}x, y -
\frac{1}{v}x\right\rangle}\right)^{p'}d_Ny\right]^{\frac{1}{p'}}\\
& \leq & \left[\int_{{\mathbb R}^d}|f(x - y)|^p|g(y)|^p
d_Ny\right]^{\frac{1}{p}} \cdot 1.
\end{eqnarray*}
Let us first raise the last inequality to the power $p$, then
integrate it with respect to $x$, and apply Fubini (actually
Tonelli) theorem. We obtain:
\begin{eqnarray*}
& \ & \int_{{\mathbb R}^d}
\left|\left[\Gamma\left(\frac{1}{\sqrt{u}}\right)\varphi \diamond
\Gamma\left(\frac{1}{\sqrt{v}}\right)\psi\right]
\left(\frac{x}{\sqrt{uv}}\right)e^{-\frac{\langle
x, x \rangle}{2puv}}\right|^pd_Nx\\
& \leq & \int_{{\mathbb R}^d}\left[\int_{{\mathbb R}^d}|f(x -
y)|^p|g(y)|^p
d_Ny\right]d_Nx\\
& = & \int_{{\mathbb R}^d} |g(y)|^p \left[\int_{{\mathbb R}^d} |f(x
- y)|^p
d_Nx\right]d_Ny\\
& = & \int_{{\mathbb R}^d} |g(y)|^p \||f|\|_p^pd_Ny\\
& = & \||f|\|_p^p\int_{{\mathbb R}^d} |g(y)|^p d_Ny\\
& = & \||f|\|_p^p \cdot \||g|\|_p^p.
\end{eqnarray*}
Raising both sides of this inequality to the power $1/p$, we get:
\begin{eqnarray*}
& \ & \left[\int_{{\mathbb R}^d}
\left|\left[\Gamma\left(\frac{1}{\sqrt{u}}\right)\varphi \diamond
\Gamma\left(\frac{1}{\sqrt{v}}\right)\psi\right]
\left(\frac{x}{\sqrt{uv}}\right)\right|^pe^{-\frac{\langle x, x
\rangle}{2uv}}d_Nx\right]^{1/p}\\
& \leq & \||f|\|_p \cdot \||g|\|_p\\
& = & \left[\int_{{\bf
R}^d}\left|\varphi\left(\frac{x}{\sqrt{v}}\right)\right|^pe^{-\frac{\langle
x, x \rangle}{2v}}d_Nx\right]^{1/p} \cdot \left[\int_{{\bf
R}^d}\left|\psi\left(\frac{x}{\sqrt{u}}\right)\right|^pe^{-\frac{\langle
x, x \rangle}{2u}}d_Nx\right]^{1/p}.
\end{eqnarray*}
Making now the changes of variable: $x_1 := (1/\sqrt{uv})x$ in the
integral
\begin{eqnarray*}
& \ & \int_{{\mathbb R}^d}
\left|\left[\Gamma\left(\frac{1}{\sqrt{u}}\right)\varphi \diamond
\Gamma\left(\frac{1}{\sqrt{v}}\right)\psi\right]
\left(\frac{x}{\sqrt{uv}}\right)\right|^pe^{-\frac{\langle x, x
\rangle}{2uv}}d_Nx,
\end{eqnarray*}
$x' := (1/\sqrt{v})x$ and $x'' := (1/\sqrt{u})x$ in the integrals:
\begin{eqnarray*}
& \ & \int_{{\bf
R}^d}\left|\varphi\left(\frac{x}{\sqrt{v}}\right)\right|^pe^{-\frac{\langle
x, x \rangle}{2v}}d_Nx
\end{eqnarray*}
and
\begin{eqnarray*}
& \ & \int_{{\bf
R}^d}\left|\psi\left(\frac{x}{\sqrt{u}}\right)\right|^pe^{-\frac{\langle
x, x \rangle}{2u}}d_Nx,
\end{eqnarray*}
respectively, and dividing both sides by $(uv)^{d/(2p)}$, since
$d\mu = \exp(-\langle x, x \rangle/2)d_Nx$, we get:
\begin{eqnarray*}
\left\|\Gamma\left(\frac{1}{\sqrt{u}}\right)\varphi \diamond
\Gamma\left(\frac{1}{\sqrt{v}}\right)\psi\right\|_p & \leq &
\|\varphi\|_p \cdot \|\psi\|_p.
\end{eqnarray*}
\end{proof}

\begin{theorem} {\bf (General H\"older inequality for
Gaussian Wick products.)} \ Let $H$ be a separable Gaussian Hilbert
space, $p \in [1$, $\infty]$, and $u$ and $v$ positive numbers, such
that $(1/u) + (1/v) = 1$. Let ${\mathcal F}(H)$ be the
sigma--algebra generated by the random variables $h$ from $H$. Then
for any $\varphi$ and $\psi$ in $L^p(\Omega$, ${\mathcal F}(H)$,
$P)$, $\Gamma(1/\sqrt{u})\varphi \diamond \Gamma(1/\sqrt{v})\psi \in
L^p(\Omega$, ${\mathcal F}(H)$, $P)$ and the following inequality
holds:
\begin{eqnarray}
\left\|\Gamma\left(\frac{1}{\sqrt{u}}\right)\varphi \diamond
\Gamma\left(\frac{1}{\sqrt{v}}\right)\psi\right\|_p & \leq &
\|\varphi\|_p \cdot \|\psi\|_p.
\end{eqnarray}
\end{theorem}
\begin{proof}
Let $\{e_n\}_{n \geq 1}$ be an orthonormal basis of $H$. Then
$\{e_n\}_{n \geq 1}$ is a set of independent, normally distributed
random variables with mean $0$ and variance $1$. For every $d \geq
1$, let ${\mathcal F}_d$ denote the sigma--algebra generated by the
random variables $e_1$, $e_2$, $\dots$, $e_d$. It is well--known
that every function $f$ from $L^p(\Omega$, ${\mathcal F}(H)$, $P)$
can be approximated in the $p$ norm by a sequence of functions $f_n$
from $L^p({\mathcal F}_n) := L^p(\Omega$, ${\mathcal F}_n$, $P)$, $n
\geq 1$. This is due to the fact that the sigma--algebra ${\mathcal
F}(H)$ is generated by the cylinder sets, and every cylinder set is
in one of the sigma--algebras ${\mathcal F}_d$, for some $d \geq 1$.
There exist two sequences $\{\varphi_n\}_{n \geq 1}$ and
$\{\psi_n\}_{n \geq 1}$ contained in $\cup_{n \geq 1}L^{p}({\mathcal
F}_n)$, such that $\|\varphi_n - \varphi\|_p \to 0$ and $\|\psi_n -
\psi\|_p \to 0$, as $n \to \infty$. These facts imply two things:
first $\|\varphi_n\|_p \to \|\varphi\|_p$ and $\|\psi_n\|_p \to
\|\psi\|_p$, as $n \to \infty$, and second $\varphi_n \to \varphi$
and $\psi_n \to \psi$, in $L^1$, as $n \to \infty$. We know from
\cite{ls2010} that $\Gamma(1/\sqrt{u})\varphi_n \diamond
\Gamma(1/\sqrt{v})\psi_n \to \Gamma(1/\sqrt{u})\varphi \diamond
\Gamma(1/\sqrt{v})\psi$, in $L^1$, as $n \to \infty$. Since $L^1$
convergence implies almost sure convergence for a subsequence,
working eventually with a subsequence, we may assume that
$\Gamma(1/\sqrt{u})\varphi_n \diamond \Gamma(1/\sqrt{v})\psi_n \to
\Gamma(1/\sqrt{u})\varphi \diamond \Gamma(1/\sqrt{v})\psi$, almost
surely, as $n \to \infty$. Using now Fatou's lemma and the finite
dimensional inequality proven in the previous theorem, we have:
\begin{eqnarray*}
\left\|\Gamma\left(\frac{1}{\sqrt{u}}\right)\varphi \diamond
\Gamma\left(\frac{1}{\sqrt{v}}\right)\psi \right\|_p^p & = &
E\left[\left|\Gamma\left(\frac{1}{\sqrt{u}}\right)\varphi \diamond
\Gamma\left(\frac{1}{\sqrt{v}}\right)\psi\right|^p\right]\\
& = & E\left[\liminf_{n \to
\infty}\left\{\left|\Gamma\left(\frac{1}{\sqrt{u}}\right)\varphi_n
\diamond \Gamma\left(\frac{1}{\sqrt{v}}\right)\psi_n\right|^p\right\}\right]\\
& \leq & \liminf_{n \to
\infty}E\left[\left|\Gamma\left(\frac{1}{\sqrt{u}}\right)\varphi_n
\diamond
\Gamma\left(\frac{1}{\sqrt{v}}\right)\psi_n\right|^p\right]\\
& \leq & \liminf_{n \to
\infty}\left\{E\left[\left|\varphi_n\right|^p\right] \cdot
E\left[\left|\psi_n\right|^p\right]\right\}.\\
& = & E\left[\left|\varphi\right|^p\right] \cdot
E\left[\left|\psi\right|^p\right]\\
& = & \|\varphi\|_p^p \cdot \|\psi\|_p^p.
\end{eqnarray*}
\end{proof}

We are now presenting an interesting connection between the Young
inequality with the best constant, and Nelson hypercontractivity.
The Young inequality with the best constant is the following:

\begin{theorem}
Let $p$, $q$, $r \geq 1$, such that $(1/p) + (1/q) = (1/r) + 1$.
There exists a constant $C_{p, q, r; d} > 0$, such that, for any $f
\in L^p\left({\mathbb R}^d, dx\right)$ and $g \in L^q\left({\mathbb
R}^d, dx\right)$, we have:
\begin{eqnarray}
\left[\int_{{\mathbb R}^d}|(f \star g)(x)|^rdx\right]^{1/r} & \leq &
C_{p, q, r; d}\left[\int_{{\mathbb R}^d}|f(x)|^pdx\right]^{1/p}
\cdot \left[\int_{{\mathbb R}^d}|g(x)|^qdx\right]^{1/q}.
\label{sharpYoung}
\end{eqnarray}
The sharp constant $C_{p, q, r; d}$ equals $(C_pC_q/C_r)^d$, where
$C_k^2 = k^{1/k}/k'^{1/k'}$, for any $k \geq 1$, where $k'$ is the
conjugate of $k$.
\end{theorem}
See \cite{b1975}, \cite{bl1976}, or \cite{ll2001} (Theorem 4.2) for
a proof. Let us make the observation that if we replace the Lebesgue
measure $dx$, on ${\mathbb R}^d$, by $cdx$, where $c$ is any
positive constant, then, by convoluting with respect to $cdx$, the
best constant $C_{p, q, r: d}$ from inequality (\ref{sharpYoung})
does not change. This is due to the fact that the left--hand side of
(\ref{sharpYoung}) is multiplied by $c \cdot c^{1/r} = c^{1 +
(1/r)}$, while the right--hand side must be multiplied by $c^{1/p}
\cdot c^{1/q} = c^{(1/p) + (1/q)}$, and fortunately $1 + (1/r) =
(1/p) + (1/q)$. Thus, the inequality (\ref{sharpYoung}) remains
valid, with the same sharp constant $C_{p, q, r; d}$, even for the
normalized Lebesgue measure $d_Nx$.

\begin{theorem}\label{Nelson} Let $d$ be a fixed natural number.
Let $p$ and $r$ be positive numbers such that $1 < p \leq r$. Then
for any $\varphi$ in $L^p({\mathbb R}^d$, $\mu)$ and any $\psi$ in
$L^{\infty}({\mathbb R}^d$, $\mu)$, $\Gamma(\sqrt{p - 1}/\sqrt{r -
1})\varphi \diamond \Gamma(\sqrt{r - p}/\sqrt{r - 1})\psi$ belongs
to $L^{r}({\mathbb R}^d$, $\mu)$ and the following inequality holds:
\begin{eqnarray}
\left\|\left[\Gamma\left(\frac{\sqrt{p - 1}}{\sqrt{r -
1}}\right)\varphi\right] \diamond \left[\Gamma\left(\frac{\sqrt{r -
p}}{\sqrt{r - 1}}\right)\psi\right]\right\|_r & \leq & \|\varphi\|_p
\cdot \|\psi\|_{\infty}. \label{N}
\end{eqnarray}
In particular, if we choose $\psi = 1$ (the constant random variable
equal to $1$), then we get the classic Nelson hypercontractivity
inequality:
\begin{eqnarray}
\left\|\Gamma\left(\frac{\sqrt{p - 1}}{\sqrt{r -
1}}\right)\varphi\right\|_r & \leq & \|\varphi\|_p.
\end{eqnarray}
\end{theorem}

\begin{proof}
Let $u := (r - 1)/(p - 1)$ and $v := (r - 1)/(r - p)$. Then we have
$(1/u) + (1/v) = 1$. Let $p'$ and $r'$ be the conjugates of $p$ and
$r$, respectively . Let's go back to the identity (\ref{convWick})
and multiply both sides of that relation by $\exp(-\langle x, x
\rangle/(2uvr'))$. We get:
\begin{eqnarray*}
& \ & \left[\Gamma\left(\frac{1}{\sqrt{u}}\right)\varphi \diamond
\Gamma\left(\frac{1}{\sqrt{v}}\right)\psi\right]
\left(\frac{x}{\sqrt{uv}}\right)e^{-\frac{\langle
x, x \rangle}{2uv}\left(1 - \frac{1}{r'}\right)}\\
& = & e^{\frac{\langle x, x \rangle}{2r'uv}}\left\{ \left[\varphi
\left(\frac{x}{\sqrt{v}}\right)e^{-\frac{\langle x, x
\rangle}{2v}}\right] \star
\left[\psi\left(\frac{x}{\sqrt{u}}\right)e^{-\frac{\langle x, x
\rangle}{2u}}\right] \right\}\\
& = & \int_{{\mathbb R}^d}\left[\varphi\left(\frac{x -
y}{\sqrt{v}}\right)e^{-\frac{\langle x - y, x - y
\rangle}{2pv}}\right]e^{-\frac{\langle x - y, x - y \rangle}{2p'v}}
\left[\psi\left(\frac{y}{\sqrt{u}}\right)e^{-\frac{\langle y, y
\rangle}{2u}}\right] \cdot e^{\frac{\langle x, x
\rangle}{2r'uv}}d_Ny.
\end{eqnarray*}
Let $f(x) := \varphi(x/\sqrt{v})\exp(-\langle x, x\rangle/(2pv))$.
We have:
\begin{eqnarray}
& \ & \left|\left[\Gamma\left(\frac{1}{\sqrt{u}}\right)\varphi
\diamond \Gamma\left(\frac{1}{\sqrt{v}}\right)\psi\right]
\left(\frac{x}{\sqrt{uv}}\right)e^{-\frac{\langle x, x
\rangle}{2uv}\left(1 - \frac{1}{r'}\right)}\right| \nonumber\\
& \leq & \int_{{\mathbb R}^d}|f(x - y)|e^{-\frac{\langle x - y, x -
y \rangle}{2p'v}}
\left|\psi\left(\frac{y}{\sqrt{u}}\right)\right|e^{-\frac{\langle y,
y \rangle}{2u}} \cdot e^{\frac{\langle x, x \rangle}{2r'uv}}d_Ny \nonumber\\
& \leq & \int_{{\mathbb R}^d}|f(x - y)|e^{-\frac{\langle x - y, x -
y \rangle}{2p'v}} \|\psi\|_{\infty}e^{-\frac{\langle y,
y \rangle}{2u}} \cdot e^{\frac{\langle x, x \rangle}{2r'uv}}d_Ny \nonumber\\
& = & \|\psi\|_{\infty}\int_{{\mathbb R}^d}|f(x - y)| \cdot
e^{-\frac{\langle x - y, x - y \rangle}{2p'v}} \cdot
e^{-\frac{\langle y, y \rangle}{2u}} \cdot
e^{\frac{\langle x, x \rangle}{2r'uv}}d_Ny \nonumber\\
& = & \|\psi\|_{\infty}\int_{{\mathbb R}^d}|f(x - y)| \cdot
e^{-\left[\frac{\langle x - y, x - y \rangle}{2p'v} + \frac{\langle
y, y \rangle}{2u} - \frac{\langle x, x \rangle}{2r'uv}\right]}d_Ny.
\label{N1}
\end{eqnarray}
Let us observe that the expression:
\begin{eqnarray*}
E(x, y) & = & \frac{\langle x - y, x - y \rangle}{2p'v} +
\frac{\langle y, y \rangle}{2u} - \frac{\langle x, x \rangle}{2r'uv}
\end{eqnarray*}
is a perfect square. Indeed, the coefficient of $\langle x$, $x
\rangle$ in $E(x$, $y)$ is:
\begin{eqnarray*}
a & = & \frac{1}{2p'v} - \frac{1}{2r'uv}\\
& = & \frac{1}{2v}\left(\frac{1}{p'} - \frac{1}{r'} \cdot
\frac{1}{u}\right)\\
& = & \frac{1}{2v}\left(\frac{p - 1}{p} - \frac{r - 1}{r} \cdot
\frac{p - 1}{r - 1}\right)\\
& = & \frac{p - 1}{2v}\left(\frac{1}{p} - \frac{1}{r}\right)\\
& = & \frac{p - 1}{2v} \cdot \frac{r - p}{pr}\\
& = & \frac{p - 1}{p} \cdot \frac{r - 1}{r} \cdot \frac{r - p}{r -
1}
\cdot \frac{1}{2v}\\
& = & \frac{1}{p'} \cdot \frac{1}{r'} \cdot \frac{1}{v} \cdot
\frac{1}{2v}\\
& = & \frac{1}{2p'r'v^2}.
\end{eqnarray*}
The coefficient of $\langle y$, $y \rangle$ in $E(x$, $y)$ is:
\begin{eqnarray*}
c & = & \frac{1}{2u} + \frac{1}{2p'v}\\
& = & \frac{1}{2}\left(\frac{1}{u} + \frac{1}{p'} \cdot
\frac{1}{v}\right)\\
& = & \frac{1}{2}\left(\frac{p - 1}{r - 1} + \frac{p - 1}{p} \cdot
\frac{r - p}{r - 1}\right)\\
& = & \frac{p - 1}{2(r - 1)}\left(1 + \frac{r - p}{p}\right)\\
& = & \frac{p - 1}{2(r - 1)} \cdot \frac{r}{p}\\
& = & \frac{1}{2} \cdot \frac{r}{r - 1} \cdot \frac{p - 1}{p}\\
& = & \frac{1}{2} \cdot r' \cdot \frac{1}{p'}\\
& = & \frac{r'}{2p'}.
\end{eqnarray*}
The coefficient of $\langle x$, $y \rangle$ is $E(x$, $y)$ is:
\begin{eqnarray*}
b & = & -\frac{1}{p'v}.
\end{eqnarray*}
Thus we have:
\begin{eqnarray*}
E(x, y) & = & a\langle x, x \rangle + b\langle x, y \rangle +
c\langle y, y \rangle\\
& = & \frac{1}{2p'r'v^2}\langle x, x \rangle - \frac{1}{p'v}\langle
x, y \rangle + \frac{r'}{2p'}\langle y, y \rangle\\
& = & \frac{1}{2p'r'v^2}\langle x - r'vy, x - r'vy\rangle.
\end{eqnarray*}
It follows now from (\ref{N1}) that:
\begin{eqnarray*}
& \ & \left|\left[\Gamma\left(\frac{1}{\sqrt{u}}\right)\varphi
\diamond \Gamma\left(\frac{1}{\sqrt{v}}\right)\psi\right]
\left(\frac{x}{\sqrt{uv}}\right)e^{-\frac{\langle x, x
\rangle}{2uv}\cdot\frac{1}{r}}\right| \\
& \leq & \|\psi\|_{\infty}\int_{{\mathbb R}^d}|f(x - y)| \cdot
e^{-\frac{1}{2p'r'v^2}\langle x - r'vy, x - r'vy\rangle}d_Ny.
\end{eqnarray*}
Let us make the change of variable $t := x - y$ in the last
integral. We obtain:
\begin{eqnarray*}
& \ & \left|\left[\Gamma\left(\frac{1}{\sqrt{u}}\right)\varphi
\diamond \Gamma\left(\frac{1}{\sqrt{v}}\right)\psi\right]
\left(\frac{x}{\sqrt{uv}}\right)e^{-\frac{\langle x, x
\rangle}{2uvr}}\right| \\
& \leq & \|\psi\|_{\infty}\int_{{\mathbb R}^d}|f(t)| \cdot
e^{-\frac{1}{2p'r'v^2}\langle x - r'v(x - t), x - r'v(x -
t)\rangle}d_Nt\\
& = & \|\psi\|_{\infty}\int_{{\mathbb R}^d}|f(t)| \cdot
e^{-\frac{1}{2p'r'v^2}\langle (1 - r'v)x + r'vt, (1 - r'v)x + r'vt
\rangle}d_Nt\\
& = & \|\psi\|_{\infty}\int_{{\mathbb R}^d}|f(t)| \cdot
e^{-\frac{1}{2p'r'v^2} \cdot (-r'v)^2\left\langle \frac{r'v -
1}{r'v}x - t,
\frac{r'v - 1}{r'v}x - t\right\rangle}d_Nt\\
& = & \|\psi\|_{\infty}\int_{{\mathbb R}^d}|f(t)| \cdot
e^{-\frac{r'}{2p'} \cdot \left\langle \frac{1}{s'}x - t,
\frac{1}{s'}x - t\right\rangle}d_Nt,
\end{eqnarray*}
where $s := r'v$ and $s'$ is the conjugate of $s$. Let us observe
that the last integral is a convolution product. Indeed, if we
define $g(x) := \exp(-[r'/(2p')] \cdot \langle x$, $x\rangle)$,
then:
\begin{eqnarray*}
& \ & \left|\left[\Gamma\left(\frac{1}{\sqrt{u}}\right)\varphi
\diamond \Gamma\left(\frac{1}{\sqrt{v}}\right)\psi\right]
\left(\frac{x}{\sqrt{uv}}\right)e^{-\frac{\langle x, x
\rangle}{2uvr}}\right| \\
& \leq & \|\psi\|_{\infty}\int_{{\mathbb R}^d}|f(t)| \cdot
e^{-\frac{r'}{2p'} \cdot \left\langle \frac{1}{s'}x - t,
\frac{1}{s'}x - t\right\rangle}d_Nt\\
& = & \|\psi\|_{\infty} \cdot \left[f \star
g\right]\left(\frac{1}{s'}x\right),
\end{eqnarray*}
for all $x \in {\mathbb R}^d$. Replacing $x$ by $s'x$, in the last
inequality, we obtain:
\begin{eqnarray}
\left|\left[\Gamma\left(\frac{1}{\sqrt{u}}\right)\varphi \diamond
\Gamma\left(\frac{1}{\sqrt{v}}\right)\psi\right]
\left(\frac{s'x}{\sqrt{uv}}\right)e^{-\frac{s'^2\langle x, x
\rangle}{2uvr}}\right| & \leq & \|\psi\|_{\infty} \cdot \left[f
\star g\right](x),
\end{eqnarray}
for all $x \in {\mathbb R}^d$. Raising this inequality to the power
$r$ and integrating with respect to the normalized Lebesgue measure
$d_Nx$, we get:
\begin{eqnarray*}
& \ & \left\{\int_{{\mathbb
R}^d}\left|\left[\Gamma\left(\frac{1}{\sqrt{u}}\right)\varphi
\diamond \Gamma\left(\frac{1}{\sqrt{v}}\right)\psi\right]
\left(\frac{s'x}{\sqrt{uv}}\right)e^{-\frac{s'^2\langle
x, x \rangle}{2uvr}}\right|^rd_Nx\right\}^{1/r}\\
& \leq & \|\psi\|_{\infty} \cdot \left\{\int_{{\mathbb R}^d}|(f
\star g)(x)|^rd_Nx\right\}^{1/r}.
\end{eqnarray*}
Making, the change of variable $x' := (s'/\sqrt{uv})x$ in the
integral from the left, we obtain:
\begin{eqnarray}
\left(\frac{\sqrt{uv}}{s'}\right)^{d/r}\left\|\left|
\left[\Gamma\left(\frac{1}{\sqrt{u}}\right)\varphi \diamond
\Gamma\left(\frac{1}{\sqrt{v}}\right)\psi\right] \cdot
e^{-\frac{1}{2r}\langle \cdot, \cdot \rangle} \right|\right\|_r &
\leq & \|\psi\|_{\infty} \cdot \|| f \star g |\|_r.
\end{eqnarray}
The left hand--side is a Gaussian $L^r$--norm, and so, we get:
\begin{eqnarray}
\left(\frac{\sqrt{uv}}{s'}\right)^{d/r}\left\|
\Gamma\left(\frac{1}{\sqrt{u}}\right)\varphi \diamond
\Gamma\left(\frac{1}{\sqrt{v}}\right)\psi\right\|_r & \leq &
\|\psi\|_{\infty} \cdot \|| f \star g |\|_r. \label{preYoung}
\end{eqnarray}
Since $r \geq p$, $(1/r) + 1 - (1/p) \leq 1$. Thus there exists $q
\geq 1$, such that $(1/r) + 1 = (1/p) + (1/q)$. We apply now the
Young inequality with the sharp constant, in the right side of
(\ref{preYoung}), and obtain:
\begin{eqnarray}
\left\| \Gamma\left(\frac{1}{\sqrt{u}}\right)\varphi \diamond
\Gamma\left(\frac{1}{\sqrt{v}}\right)\psi\right\|_r & \leq &
\left(\frac{s'}{\sqrt{uv}}\right)^{d/r} \cdot \|\psi\|_{\infty}
\cdot \||f \star g|\|_r \nonumber\\
& \leq & \left(\frac{s'}{\sqrt{uv}}\right)^{d/r} \cdot
\|\psi\|_{\infty} \cdot (C_pC_q/C_r)^d \||f|\|_p \cdot \||g|\|_q,
\label{N2}
\end{eqnarray}
where $C_p^2 = p^{1/p}/p'^{1/p'}$. Since $f(x) =
\varphi(x/\sqrt{v})\exp(-\langle x, x \rangle/(2pv))$, it is easy to
see that:
\begin{eqnarray}
\||f|\|_p & = & (\sqrt{v})^{d/p}\|\varphi\|_p.
\end{eqnarray}
Because $g(x) := \exp(-[r'/(2p')] \cdot \langle x$, $x\rangle)$, it
is not hard to see that:
\begin{eqnarray*}
\||g|\|_q & = & \left(\sqrt{\frac{p'}{qr'}}\right)^{d/q}.
\end{eqnarray*}
Thus, inequality (\ref{N2}) becomes:
\begin{eqnarray*}
\left\| \Gamma\left(\frac{1}{\sqrt{u}}\right)\varphi \diamond
\Gamma\left(\frac{1}{\sqrt{v}}\right)\psi\right\|_r & \leq &
\left(\frac{s'}{\sqrt{uv}}\right)^{d/r}(C_pC_q/C_r)^d
(\sqrt{v})^{d/p}\left(\sqrt{\frac{p'}{qr'}}\right)^{d/q}
\|\varphi\|_p\|\psi\|_{\infty}.
\end{eqnarray*}
Therefore, to prove (\ref{N}), we just need to show that:
\begin{eqnarray}
\left(\frac{s'}{\sqrt{uv}}\right)^{d/r} \cdot (C_pC_q/C_r)^d \cdot
(\sqrt{v})^{d/p} \cdot \left(\sqrt{\frac{p'}{qr'}}\right)^{d/q} & =
& 1,
\end{eqnarray}
which (by raising both sides to the power $2/d$) is equivalent to:
\begin{eqnarray*}
\frac{C_p^2C_q^2s'^{2/r}p'^{1/q}v^{1/p -
1/r}}{C_r^2u^{1/r}q^{1/q}r'^{1/q}} & = & 1.
\end{eqnarray*}
Since $1/p - 1/r = 1 - 1/q$ and $1 - 1/q = 1/q'$, we have to prove
that:
\begin{eqnarray}
\frac{C_p^2C_q^2s'^{2/r}p'^{1/q}v^{1/q'}}
{C_r^2u^{1/r}q^{1/q}r'^{1/q}} & = & 1. \label{Nc=1}
\end{eqnarray}
Let:
\begin{eqnarray}
C & := & \frac{C_p^2C_q^2s'^{2/r}p'^{1/q}v^{1/q'}}
{C_r^2u^{1/r}q^{1/q}r'^{1/q}}. \label{sNc=1}
\end{eqnarray}
To prove (\ref{Nc=1}) we will write $u$, $v$, and $s'$ in terms of
$p$, $q$, $r$ and their conjugates. We have:
\begin{eqnarray}
u & = & \frac{r - 1}{p - 1} \nonumber\\
& = & \frac{r}{p} \cdot \frac{1 - \frac{1}{r}}{1 - \frac{1}{p}} \nonumber\\
& = & \frac{r}{p} \cdot \frac{\frac{1}{r'}}{\frac{1}{p'}} \nonumber\\
& = & \frac{rp'}{pr'}. \label{u}
\end{eqnarray}
We also have:
\begin{eqnarray*}
v & = & \frac{r - 1}{r - p}\\
& = & \frac{r\left(1 - \frac{1}{r}\right)}{pr\left(\frac{1}{p} -
\frac{1}{r}\right)}\\
& = & \frac{1}{p} \cdot \frac{1 - \frac{1}{r}}{\frac{1}{p} -
\frac{1}{r}}.
\end{eqnarray*}
Let's remember that $(1/p) - (1/r) = 1/q'$. Thus, we obtain:
\begin{eqnarray}
v & = & \frac{1}{p} \cdot \frac{1 - \frac{1}{r}}{\frac{1}{p} -
\frac{1}{r}} \nonumber\\
& = & \frac{1}{p} \cdot \frac{\frac{1}{r'}}{\frac{1}{q'}} \nonumber\\
& = & \frac{q'}{pr'}. \label{v}
\end{eqnarray}
Finally, we have:
\begin{eqnarray*}
s' & = & \frac{s}{s - 1}\\
& = & \frac{r'v}{r'v - 1}.
\end{eqnarray*}
Replacing $v$ by $q'/(pr')$ we get:
\begin{eqnarray*}
s' & = & \frac{r'v}{r'v - 1} \\
& = & \frac{r'\frac{q'}{pr'}}{r'\frac{q'}{pr'} - 1} \\
& = & \frac{\frac{q'}{p}}{\frac{q'}{p} - 1}.
\end{eqnarray*}
Dividing both the numerator and denominator of the last fraction by
$q'$ we get:
\begin{eqnarray}
s' & = & \frac{\frac{1}{p}}{\frac{1}{p} - \frac{1}{q'}} \nonumber\\
& = & \frac{\frac{1}{p}}{\frac{1}{p} - \left(1 - \frac{1}{q}\right)} \nonumber\\
& = & \frac{\frac{1}{p}}{\frac{1}{p} + \frac{1}{q} - 1} \nonumber\\
& = & \frac{\frac{1}{p}}{\frac{1}{r}} \nonumber\\
& = & \frac{r}{p}. \label{s'}
\end{eqnarray}
Let us substitute $C_p^2$, $C_q^2$, $C_r^2$, $u$, $v$, and $s'$, in
the formula (\ref{sNc=1}). We have:
\begin{eqnarray*}
C & = & \frac{C_p^2C_q^2s'^{2/r}p'^{1/q}v^{1/q'}}
{C_r^2u^{1/r}q^{1/q}r'^{1/q}}\\
& = & \frac{\frac{p^{1/p}}{p'^{1/p'}} \cdot
\frac{q^{1/q}}{q'^{1/q'}} \cdot \left(\frac{r}{p}\right)^{2/r} \cdot
p'^{1/q}\cdot
\left(\frac{q'}{pr'}\right)^{1/q'}}{\frac{r^{1/r}}{r'^{1/r'}} \cdot
\left(\frac{rp'}{pr'}\right)^{1/r} \cdot q^{1/q} \cdot r'^{1/q}}.
\end{eqnarray*}
Observe that the factors $q^{1/q}$, $q'^{1/q'}$, and $r^{2/r}$
cancel. Collecting the powers of $p$, $p'$, and $r'$, we obtain:
\begin{eqnarray*}
C & = & p^{\frac{1}{p} - \frac{1}{r} - \frac{1}{q'}} \cdot
p'^{\frac{1}{q} - \frac{1}{p'} - \frac{1}{r}} \cdot r'^{\frac{1}{r'}
+ \frac{1}{r} - \frac{1}{q'} - \frac{1}{q}}\\
& = & p^{\frac{1}{p} - \frac{1}{r} - \left(1 - \frac{1}{q}\right)}
\cdot p'^{\frac{1}{q} - \left(1 - \frac{1}{p}\right) - \frac{1}{r}}
\cdot r'^{\left(\frac{1}{r'} + \frac{1}{r}\right) -
\left(\frac{1}{q'} + \frac{1}{q}\right)}\\
& = & p^{\left(\frac{1}{p} + \frac{1}{q}\right) - \left(\frac{1}{r}
+ 1\right)} \cdot p'^{\left(\frac{1}{p} + \frac{1}{q}\right) -
\left(\frac{1}{r} + 1\right)} \cdot r'^{1 - 1}\\
& = & p^0 \cdot p'^0 \cdot r'^0\\
& = & 1.
\end{eqnarray*}
\end{proof}

Passing from the finite dimensional case to the infinite dimensional
case, can be done in the same way as before, using Fatou's lemma.\\

\par We now present a more general H\"older inequality. The proof of
this inequality uses the following theorem of Lieb (see \cite{l1990}
or \cite{ll2001} (page 100)).

\begin{theorem} \label{Lieb}
Fix $k > 1$, integers $n_1$, $\dots$, $n_k$ and numbers $p_1$,
$\dots$, $p_k \geq 1$. Let $M \geq 1$ and let $B_i$ (for $i = 1$,
$\dots$, $k$) be a linear mapping from ${\mathbb R}^M$ to ${\mathbb
R}^{n_i}$. Let $Z : {\mathbb R}^M \to {\mathbb R}^+$ be some fixed
Gaussian function,
\begin{eqnarray*}
Z(x) & = & \exp\left[-\langle x, Jx \rangle\right]
\end{eqnarray*}
with $J$ a real, positive--semidefinite $M \times M$ matrix
(possible zero).
\par For functions $f_i$ in $L^{p_i}({\mathbb R}^{n_i})$ consider the
integral $I_Z$ and the number $C_Z$:
\begin{eqnarray}
I_Z(f_1, \dots, f_k) & = & \int_{{\mathbb R}^M}Z(x)\prod_{i =
1}^kf_i(B_ix)dx \label{I_Z}
\end{eqnarray}
\begin{eqnarray}
C_Z & := & \sup\{I_Z(f_1, \dots, f_k) \mid \||f_i|\|_{p_i} = 1 \ for
\ i = 1, \dots, k\}. \label{C_Z}
\end{eqnarray}
Then $C_Z$ is determined by restricting the $f$'s to be Gaussian
functions, i.e.,
\begin{eqnarray*}
C_Z & = & \sup\{I_Z(f_1, \dots, f_k) \mid \||f_i|\|_{p_i} = 1 \ and
\ f_i(x) =
c_i\exp[-\langle x, J_ix \rangle]\\
& \ & with \ c_i > 0, \ and \ J_i \ a \ real, \ symmetric, \
positive-definite \ n_i \times n_i \ matrix\}.
\end{eqnarray*}
\end{theorem}

\begin{corollary} \label{corlieb}
Let $p$, $q$, $r \geq 1$. Let $B_1$ and $B_2$ be linear maps from
${\mathbb R}^2$ to ${\mathbb R}^2$, and $J$ a real,
positive--semidefinite $2 \times 2$ matrix (possible zero). For $f$
in $L^p({\mathbb R}^2)$ and $g$ in $L^q({\mathbb R}^2)$, we consider
the product:
\begin{eqnarray}
\left(f \star_{B_1, B_2, J} g\right)(x) & = & \int_{{\mathbb
R}}f(B_1(x, y))g(B_2(x, y))e^{-\langle (x, y), J(x, y)\rangle}d_Ny.
\end{eqnarray}
We define:
\begin{eqnarray}
C & := & \sup\{\|| f \star_{B_1, B_2, J} g |\|_r \mid \||f |\|_p =
\|| g |\|_q = 1\}.
\end{eqnarray}
Then $C$ is determined by restricting $f$ and $g$ to be Gaussian
functions.

\begin{proof}
Let $r'$ be the conjugate of $r$. For any $k \geq 1$, we denote by
${\mathcal G}_k$ the set of Gaussian functions of $L^k$--norm equal
to $1$. Using the duality between $L^r$ and $L^{r'}$, Lieb's
theorem, and H\"older inequality, we have:
\begin{eqnarray*}
C & = & \sup\{\|| f \star_{B_1, B_2, J} g |\|_r \mid \||f |\|_p =
\|| g |\|_q = 1\}\\
& = & \sup\left\{\left|\int_{{\mathbb R}}(f \star_{B_1, B_2, J}
g)(x)h(x)d_Nx\right|
\left| \right. \||f |\|_p = \|| g |\|_q = \||h|\|_{r'} = 1\right\}\\
& = & \sup\left\{\left|\int_{{\mathbb R}}\int_{{\mathbb R}}f(B_1(x,
y))g(B_2(x, y))h(x)e^{-\langle (x, y), J(x,
y)\rangle}d_Nxd_Ny\right|
\left|\right.\right.\\
& \ & \ \ \ \ \ \ \ \left.\||f |\|_p = \|| g |\|_q =
\||h|\|_{r'} = 1\right\}\\
& \leq & \sup\left\{\left|\int_{{\mathbb R}}\int_{{\mathbb
R}}f(B_1(x, y))g(B_2(x, y))h(x)e^{-\langle (x, y), J(x,
y)\rangle}d_Nxd_Ny\right| \left|\right. f \in {\mathcal G}_p, g \in
{\mathcal G}_q,
h \in {\mathcal G}_{r'}\right\}\\
& \leq & \sup\left\{\||f \star_{B_1, B_2, J} g|\|_r \cdot
\||h|\|_{r'} \mid f \in {\mathcal G}_p, g \in {\mathcal G}_q,
h \in {\mathcal G}_{r'}\right\}\\
& = & \sup\left\{\||f \star_{B_1, B_2, J} g|\|_r \cdot 1 \mid f \in
{\mathcal G}_p, g \in {\mathcal G}_q\right\}.
\end{eqnarray*}
\end{proof}

\end{corollary}


\begin{theorem} \label{fullholdertheorem} {\bf (Full H\"older
inequality for Gaussian Hilbert spaces.)} Let $H$ be a separable
Gaussian Hilbert space. Let $u$, $v$, $p$, $q$, and $r$ be numbers
greater than $1$,  such that:
\begin{eqnarray*}
\frac{1}{u} + \frac{1}{v} & = & 1
\end{eqnarray*}
and
\begin{eqnarray}
\frac{1}{u(p - 1)} + \frac{1}{v(q - 1)} & = & \frac{1}{r - 1}.
\label{cond1}
\end{eqnarray}
Then for all $\varphi$ in $L^p(\Omega$, ${\mathcal F}(H)$, $P)$ and
$\psi$ in $L^q(\Omega$, ${\mathcal F}(H)$, $P)$,
$\Gamma(1/\sqrt{u})\varphi \diamond \Gamma(1/\sqrt{v})\psi$ belongs
to $L^r(\Omega$, ${\mathcal F}(H)$, $P)$ and the following
inequality holds:
\begin{eqnarray}
\left\|\Gamma(1/\sqrt{u})\varphi \diamond
\Gamma(1/\sqrt{v})\psi\right\|_r & \leq & \|\varphi\|_p \cdot
\|\psi\|_q. \label{fullholder}
\end{eqnarray}
\end{theorem}
\begin{proof} Let $p'$, $q'$, and $r'$ be the conjugates of $p$,
$q$, and $r$, respectively. Since:
\begin{eqnarray*}
\frac{1}{p - 1} & = & \frac{p}{p - 1} - 1\\
& = & p' - 1
\end{eqnarray*}
and similarly $1/(q - 1) = q' - 1$, and $1/(r - 1) = r' - 1$,
condition (\ref{cond1}) is equivalent to:
\begin{eqnarray*}
r' - 1 & = & \frac{1}{r - 1}\\
& = & \frac{1}{u(p - 1)} + \frac{1}{v(q - 1)}\\
& = & \frac{1}{u}\left(p' - 1\right) + \frac{1}{v}\left(q' -
1\right)\\
& = & \frac{1}{u} \cdot p' + \frac{1}{v} \cdot q' -
\left(\frac{1}{u} + \frac{1}{v}\right)\\
& = & \frac{1}{u} \cdot p' + \frac{1}{v} \cdot q' - 1.
\end{eqnarray*}
That means, we have:
\begin{eqnarray}
\frac{p'}{u} + \frac{q'}{v} & = & r'. \label{cond2}
\end{eqnarray}

Following the same steps as before, it is enough to check the
inequality in the finite dimensional case. Multiplying both sides of
the convolution identity (\ref{convWick}) by $\exp(1/(2uvr')\langle
x, x \rangle)$, we obtain:
\begin{eqnarray*}
& \ & \left[\Gamma\left(\frac{1}{\sqrt{u}}\right)\varphi \diamond
\Gamma\left(\frac{1}{\sqrt{v}}\right)\psi\right]
\left(\frac{x}{\sqrt{uv}}\right)e^{-\frac{\langle
x, x \rangle}{2uv}\left(1 - \frac{1}{r'}\right)}\\
& = & e^{\frac{\langle x, x \rangle}{2r'uv}} \left\{\left[\varphi
\left(\frac{x}{\sqrt{v}}\right)e^{-\frac{\langle x, x
\rangle}{2v}}\right] \star
\left[\psi\left(\frac{x}{\sqrt{u}}\right)e^{-\frac{\langle x, x
\rangle}{2u}}\right]\right\}\\
& = & \int_{{\mathbb R}^d}\left[\varphi\left(\frac{x -
y}{\sqrt{v}}\right)e^{-\frac{\langle x - y, x - y
\rangle}{2pv}}\right]e^{-\frac{\langle x - y, x - y \rangle}{2p'v}}
\left[\psi\left(\frac{y}{\sqrt{u}}\right)e^{-\frac{\langle y, y
\rangle}{2qu}}\right]e^{-\frac{\langle y, y \rangle}{2q'u}} \cdot
e^{\frac{\langle x, x \rangle}{2r'uv}}d_Ny.
\end{eqnarray*}
Let $f(x) := \varphi(x/\sqrt{v})\exp(-\langle x, x\rangle/(2pv))$
and $g(x) := \psi(x/\sqrt{u})\exp(-\langle x, x\rangle/(2qu))$. We
have:
\begin{eqnarray}
& \ & \left|\left[\Gamma\left(\frac{1}{\sqrt{u}}\right)\varphi
\diamond \Gamma\left(\frac{1}{\sqrt{v}}\right)\psi\right]
\left(\frac{x}{\sqrt{uv}}\right)e^{-\frac{\langle x, x
\rangle}{2uv} \cdot \frac{1}{r}}\right| \nonumber\\
& \leq & \int_{{\mathbb R}^d}|f(x - y)|e^{-\frac{\langle x - y, x -
y \rangle}{2p'v}}|g(y)|e^{-\frac{\langle y,
y \rangle}{2q'u}} \cdot e^{\frac{\langle x, x \rangle}{2r'uv}}d_Ny \nonumber\\
& = & \int_{{\mathbb R}^d}|f(x - y)| \cdot |g(y)| \cdot
e^{-\frac{\langle x - y, x - y \rangle}{2p'v} - \frac{\langle y, y
\rangle}{2q'u} + \frac{\langle x, x \rangle}{2r'uv}}d_Ny. \label{L1}
\end{eqnarray}
As before, we are now showing that the expression:
\begin{eqnarray*}
E(x, y) & = & \frac{\langle x - y, x - y \rangle}{2p'v} +
\frac{\langle y, y \rangle}{2q'u} - \frac{\langle x, x
\rangle}{2r'uv}
\end{eqnarray*}
is a perfect square. Indeed, the coefficient of $\langle x$, $x
\rangle$ in $E(x$, $y)$ is:
\begin{eqnarray*}
a & = & \frac{1}{2p'v} - \frac{1}{2r'uv}\\
& = & \frac{1}{2p'r'v}\left(r' - \frac{p'}{u}\right)\\
{\rm from} \
(\ref{cond2}) & = & \frac{1}{2p'r'v} \cdot \frac{q'}{v}\\
& = & \frac{1}{2p'q'r'} \cdot \left(\frac{q'}{v}\right)^2.
\end{eqnarray*}
The coefficient of $\langle y$, $y \rangle$ in $E(x$, $y)$ is:
\begin{eqnarray*}
c & = & \frac{1}{2q'u} + \frac{1}{2p'v}\\
& = & \frac{1}{2p'q'}\left(\frac{p'}{u} + \frac{q'}{v}\right)\\
{\rm from} \
(\ref{cond2}) & = & \frac{1}{2p'q'} \cdot r'\\
& = & \frac{1}{2p'q'r'} \cdot r'^2.
\end{eqnarray*}
The coefficient of $\langle x$, $y \rangle$ is $E(x$, $y)$ is:
\begin{eqnarray*}
b & = & -\frac{1}{p'v}\\
& = & -\frac{1}{p'q'r'} \cdot \left(\frac{q'}{v} \cdot r'\right).
\end{eqnarray*}
Thus we have:
\begin{eqnarray*}
E(x, y) & = & a\langle x, x \rangle + b\langle x, y \rangle +
c\langle y, y \rangle\\
& = & \frac{1}{2p'q'r'}\left[\left(\frac{q'}{v}\right)^2\langle x, x
\rangle - 2\left(\frac{q'}{v} \cdot r'\right)\langle
x, y \rangle + r'^2\langle y, y \rangle\right]\\
& = & \frac{1}{2p'q'r'}\left\langle \frac{q'}{v}x - r'y,
\frac{q'}{v}x - r'y\right\rangle.
\end{eqnarray*}
It follows now from (\ref{L1}) that:
\begin{eqnarray}
& \ & \left|\left[\Gamma\left(\frac{1}{\sqrt{u}}\right)\varphi
\diamond \Gamma\left(\frac{1}{\sqrt{v}}\right)\psi\right]
\left(\frac{x}{\sqrt{uv}}\right)e^{-\frac{\langle x, x
\rangle}{2uvr}}\right| \nonumber\\
& \leq & \int_{{\mathbb R}^d}|f(x - y)| \cdot |g(y)| \cdot
e^{-\frac{1}{2p'q'r'}\left\langle \frac{q'}{v}x - r'y, \frac{q'}{v}x
- r'y\right\rangle}d_Ny. \label{L2}
\end{eqnarray}
{\bf Claim 1:} \ For all $d \geq 1$, $f \in L^p({\mathbb R}^d$,
$d_Nx)$, and $g \in L^q({\mathbb R}^d$, $d_Ny)$, we have:
\begin{eqnarray}
\left\{\int_{{\mathbb R}^d}\left[\int_{{\mathbb R}^d}|f(x - y)|
\cdot |g(y)| \cdot J_d(x, y)d_Ny\right]^rd_Nx\right\}^{1/r} & \leq &
C^d\||f|\|_p \cdot \||g|\|_q, \label{neededone}
\end{eqnarray}
where:
\begin{eqnarray*}
J_d(x, y) & = & e^{-\frac{1}{2p'q'r'}\left\langle \frac{q'}{v}x -
r'y, \frac{q'}{v}x - r'y\right\rangle}
\end{eqnarray*}
and
\begin{eqnarray*}
C^2 & = & v^{\frac{1}{r} - \frac{1}{p}}u^{\frac{1}{r} -
\frac{1}{q}}.
\end{eqnarray*}
To prove this claim, we reduce the problem to the one--dimensional
case, via Minkowski's inequality, copying the argument from
\cite{ll2001} (page 201). Namely, let us assume that
(\ref{neededone}) holds for two dimensions $d_1 = m$ and $d_2 = n$.
We can prove that (\ref{neededone}) holds for $d = m + n$, using
Minkowski's inequality in the form in which the discrete summation
is replaced by the continuous integration, in the following way. Let
$x = (x_m$, $x_n)$ be a generic vector in ${\mathbb R}^{m + n}$,
where $x_m$ and $x_n$ are generic vectors in ${\mathbb R}^m$ and
${\mathbb R}^n$, respectively. Let us observe that, for all $x =
(x_m$, $x_n)$ and $y = (y_m$, $y_n)$ in ${\mathbb R}^{m + n}$, we
have:
\begin{eqnarray*}
J_{m + n}(x, y) & = & J_m(x_m, y_m) \cdot J_n(x_n, y_n).
\end{eqnarray*}
We have:
\begin{eqnarray}
& \ & \left\{\int_{{\mathbb R}^{m + n}}\left[\int_{{\mathbb R}^{m +
n}}|f(x - y)| \cdot |g(y)| \cdot J_{m + n}(x,
y)d_Ny\right]^rd_Nx\right\}^{1/r} \nonumber\\
& = & \left\{\int_{{\mathbb R}^{m}}\int_{{\mathbb
R}^{n}}\left[\int_{{\mathbb R}^{m}}\int_{{\mathbb R}^{n}}|f(x_m -
y_m, x_n - y_n)| \cdot |g(y_m,
y_n)|\right.\right. \nonumber\\
& \ & \ \ \ \ \ \ \ \ \ \ \ \ \ \ \times \left.\left.J_{m +
n}\left((x_m, x_n), (y_m,
y_n)\right)d_Ny_nd_Ny_m\right]^rd_Nx_nd_Nx_m\right\}^{1/r} \nonumber\\
& \leq & \left\{\int_{{\mathbb R}^{m}}\left\{\int_{{\mathbb
R}^{m}}J_m(x_m, y_m)\left[\int_{{\mathbb R}^{n}}\left(\int_{{\mathbb
R}^{n}}|f(x_m - y_m, x_n - y_n)| \cdot |g(y_m,
y_n)|\right.\right.\right.\right. \nonumber\\
& \ & \ \ \ \ \ \ \ \ \ \ \ \ \ \ \ \ \ \ \ \ \ \ \ \ \ \ \ \ \ \ \
\ \ \ \ \ \ \times \left.\left.\left.\left.J_n(x_n,
y_n)d_Ny_n\right)^rd_Nx_n\right]^{1/r}d_Ny_m\right\}^rd_Nx_m\right\}^{1/r}
\label{Minkowski}\\
& \leq & \left\{\int_{{\mathbb R}^{m}}\left[\int_{{\mathbb
R}^{m}}J_m(x_m, y_m)C^n\||f(x_m - y_m, \cdot)|\|_p \cdot \||g(y_m,
\cdot)|\|_q
d_Ny_m\right]^rd_Nx_m\right\}^{1/r} \label{M1}\\
& = & C^n\left\{\int_{{\mathbb R}^{m}}\left[\int_{{\mathbb
R}^{m}}J_m(x_m, y_m)\||f(x_m - y_m, \cdot)|\|_p \cdot \||g(y_m,
\cdot)|\|_q
d_Ny_m\right]^rd_Nx_m\right\}^{1/r} \nonumber\\
& \leq & C^n \cdot C^m\||f|\|_p \cdot \||g|\|_q \label{M2}\\
& = & C^{m + n}\||f|\|_p \cdot \||g|\|_q. \nonumber
\end{eqnarray}
This shows that in order to prove (\ref{neededone}), it is enough to
prove it for the dimension $d = 1$ only. To achieve this, since the
function $(x$, $y) \mapsto [(q'/v)x - r'y]^2$ is non--negative,
according to Lieb theorem, it is enough to check it for exponential
functions of the form $f(x) = c_1\exp[-(s/2)x^2]$ and $g(x) =
c_2\exp[-(t/2)x^2]$, where $s > 0$, $t > 0$, and $c_1$ and $c_2$ are
positive constants chosen such that $\||f|\|_p = \||g|\|_q = 1$. Let
us first compute the values of $c_1$ and $c_2$. We have:
\begin{eqnarray*}
\||f|\|_p & = & c_1\left[\int_{{\mathbb R}}e^{-\frac{ps}{2}x^2}d_Nx\right]^{1/p}\\
({\rm let} \ x' := \sqrt{ps}x) \quad & = &
c_1\left[\frac{1}{\sqrt{ps}}\int_{{\mathbb
R}}e^{-\frac{x'^2}{2}}d_Nx'
\right]^{1/p}\\
& = & c_1\frac{1}{(\sqrt{ps})^{1/p}}.
\end{eqnarray*}
Thus, in order to have $\||f|\|_p = 1$, we must have:
\begin{eqnarray}
c_1 & = & (\sqrt{ps})^{\frac{1}{p}}. \label{c1}
\end{eqnarray}
Similarly, in order to have $\||g\||_q = 1$, we must have:
\begin{eqnarray}
c_2 & = & (\sqrt{qt})^{\frac{1}{q}}. \label{c2}
\end{eqnarray}
Hence, we have:
\begin{eqnarray*}
& \ & \left\|\left|\int_{{\mathbb R}}f(\cdot -
y)g(y)e^{-\left(\frac{q'}{v}\cdot -
r'y\right)^2}d_Ny\right|\right\|_r\\
& = &
\left(\sqrt{ps}\right)^{1/p}\left(\sqrt{qt}\right)^{1/q}\left\{\int_{{\mathbb
R}}\left[\int_{\mathbb R}e^{-\frac{s}{2}(x - y)^2}
e^{-\frac{t}{2}y^2}e^{-\frac{1}{2p'q'r'}\left(\frac{q'}{v}x -
r'y\right)^2}d_Ny\right]^rd_Nx\right\}^{1/r}.
\end{eqnarray*}
Let $\alpha := q'/(v\sqrt{p'q'r'})$, $\beta := r'/\sqrt{p'q'r'}$,
and $\gamma := p'/(u\sqrt{p'q'r'})$. Let us observe first that
$\alpha + \gamma = \beta$, since $(p'/u) + (q'/v) = r'$. We have:
\begin{eqnarray*}
& \ & \left\|\left|\int_{{\mathbb R}}f(\cdot -
y)g(y)e^{-\left(\frac{q'}{v}\cdot -
r'y\right)^2}d_Ny\right|\right\|_r\\
& = & (\sqrt{ps})^{1/p}(\sqrt{qt})^{1/q}\left\{\int_{{\mathbb
R}}\left[\int_{\mathbb R}e^{-\frac{s}{2}(x -
y)^2}e^{-\frac{t}{2}y^2}e^{-\frac{1}{2}(\alpha x
- \beta y)^2}d_Ny\right]^rd_Nx\right\}^{1/r}\\
& = & (\sqrt{ps})^{1/p}(\sqrt{qt})^{1/q}\left\{\int_{{\mathbb
R}}\left[e^{-\frac{\left(s + \alpha^2\right)}{2}x^2}\int_{\mathbb
R}e^{-\frac{\left(s + t + \beta^2\right)}{2}y^2 + (s +
\alpha\beta)xy}d_Ny\right]^rd_Nx\right\}^{1/r}.
\end{eqnarray*}
In the last integral we make the change of variable $y' = \sqrt{s +
t + \beta^2} \cdot y$. Completing the square, we obtain:
\begin{eqnarray*}
& \ & \left\|\left|\int_{{\mathbb R}}f(\cdot -
y)g(y)e^{-\left(\frac{q'}{v}\cdot -
r'y\right)^2}d_Ny\right|\right\|_r\\
& = & (\sqrt{ps})^{1/p}(\sqrt{qt})^{1/q}\left\{\int_{{\mathbb
R}}\left[e^{-\frac{\left(s + \alpha^2\right)}{2}x^2}\frac{1}{\sqrt{s
+ t + \beta^2}}\int_{\mathbb R}e^{-\frac{1}{2}y^2 + \frac{s +
\alpha\beta}{\sqrt{s + t +
\beta^2}}xy}d_Ny\right]^rd_Nx\right\}^{1/r}\\
& = & \sqrt{\frac{(ps)^{1/p}(qt)^{1/q}}{s + t +
\beta^2}}\left\{\int_{{\mathbb R}}\left[e^{-\frac{\left(s +
\alpha^2\right)}{2}x^2} \cdot e^{\frac{(s + \alpha\beta)^2}{2(s + t
+ \beta^2)}x^2}\int_{\mathbb R}e^{-\frac{1}{2}\left(y - \frac{s +
\alpha\beta}{\sqrt{s + t +
\beta^2}}x\right)^2}d_Ny\right]^rd_Nx\right\}^{1/r}.
\end{eqnarray*}
Therefore,
\begin{eqnarray*}
& \ & \left\|\left|\int_{{\mathbb R}}f(\cdot -
y)g(y)e^{-\left(\frac{q'}{v}\cdot -
r'y\right)^2}d_Ny\right|\right\|_r\\
& = & \sqrt{\frac{(ps)^{1/p}(qt)^{1/q}}{s + t +
\beta^2}}\left\{\int_{{\mathbb R}}\left[e^{-\frac{\left(s +
\alpha^2\right)}{2}x^2 + \frac{(s + \alpha\beta)^2}{2(s + t
+ \beta^2)}x^2}\right]^rd_Nx\right\}^{1/r}\\
& = &\sqrt{\frac{(ps)^{1/p}(qt)^{1/q}}{s + t +
\beta^2}}\left\{\int_{{\mathbb R}}\left[e^{-\frac{\left(s +
\alpha^2\right)\left(s + t + \beta^2\right) + (s +
\alpha\beta)^2}{2(s + t +
\beta^2)}x^2}\right]^rd_Nx\right\}^{1/r}\\
& = & \sqrt{\frac{(ps)^{1/p}(qt)^{1/q}}{s + t +
\beta^2}}\left\{\int_{{\mathbb R}}\left[e^{-\frac{\left[s(\beta -
\alpha)^2 + t\alpha^2 + st\right]}{2(s + t +
\beta^2)}x^2}\right]^rd_Nx\right\}^{1/r}\\
& = & \sqrt{\frac{(ps)^{1/p}(qt)^{1/q}}{s + t +
\beta^2}}\left\{\int_{{\mathbb R}}e^{-\frac{r\left(s\gamma^2 +
t\alpha^2 + st\right)}{2(s + t + \beta^2)}x^2}d_Nx\right\}^{1/r}\\
& = & \sqrt{\frac{(ps)^{1/p}(qt)^{1/q}}{s + t + \beta^2}}\cdot
\sqrt{\frac{(s + t + \beta^2)^{1/r}}{r^{1/r}(s\gamma^2 + t\alpha^2 +
st)^{1/r}}}\\
& = & \sqrt{\frac{p^{1/p}q^{1/q}}{r^{1/r}}} \cdot
\sqrt{\frac{s^{1/p}t^{1/q}}{(s + t + \beta^2)^{1/r'}(\gamma^2s +
\alpha^2t + st)^{1/r}}}.
\end{eqnarray*}
To finish our proof we need to show that:
\begin{eqnarray}
\sup_{s > 0, t > 0}\left\{\frac{p^{1/p}q^{1/q}}{r^{1/r}} \cdot
\frac{s^{1/p}t^{1/q}}{(s + t + \beta^2)^{1/r'}(\gamma^2s + \alpha^2t
+ st)^{1/r}}\right\} & = & v^{\frac{1}{r} -
\frac{1}{p}}u^{\frac{1}{r} - \frac{1}{q}}. \label{sup=1}
\end{eqnarray}
Before we compute this supremum, we would like to outline the
intuition behind what we are going to do next. Let us observe that
the numerator $s^{1/p}t^{1/q}$, being a product, is somehow like a
geometric mean, while the factors from the denominator $(s + t +
\beta^2)$ and $(\gamma^2s + \alpha^2t + st)$, being sums, are like
arithmetic means. We know from the inequality between the geometric
and arithmetic means of positive numbers, that the geometric mean is
always dominated by the arithmetic mean, and this classic inequality
is based on the concavity of the logarithmic function.\\
Let $S := (pv)s$ and $T := (qu)t$. We have:
\begin{eqnarray}
& \ & \frac{p^{1/p}q^{1/q}}{r^{1/r}} \cdot \frac{s^{1/p}t^{1/q}}{(s
+ t + \beta^2)^{1/r'}(\gamma^2s + \alpha^2t + st)^{1/r}} \nonumber\\
& = & \frac{p^{1/p}q^{1/q}}{r^{1/r}} \cdot
\frac{\frac{1}{(pv)^{1/p}}S^{1/p}\frac{1}{(qu)^{1/q}}
T^{1/q}}{\left(\frac{1}{pv} \cdot S + \frac{1}{qu} \cdot T + \beta^2
\cdot 1\right)^{1/r'}\left(\frac{\gamma^2}{pv}\cdot S +
\frac{\alpha^2}{qu}\cdot T + \frac{1}{pquv} \cdot ST\right)^{1/r}} \nonumber\\
& = & \frac{1}{v^{1/p}u^{1/q}r^{1/r}} \cdot
\frac{S^{1/p}T^{1/q}}{\left(\frac{1}{pv} \cdot S + \frac{1}{qu}
\cdot T + \beta^2 \cdot
1\right)^{1/r'}\left(\frac{\gamma^2}{pv}\cdot S +
\frac{\alpha^2}{qu}\cdot T + \frac{1}{pquv} \cdot ST\right)^{1/r}}.
\label{formula}
\end{eqnarray}
Let us observe that:
\begin{eqnarray}
\frac{1}{pv} + \frac{1}{qu} + \beta^2 & = & 1. \label{J1}
\end{eqnarray}
Indeed, we have:
\begin{eqnarray*}
\frac{1}{pv} + \frac{1}{qu} + \beta^2 & = & \frac{1}{pv} +
\frac{1}{qu} + \frac{r'}{p'q'}\\
& = & \frac{1}{pv} + \frac{1}{qu} + \frac{\frac{p'}{u} +
\frac{q'}{v}}{p'q'}\\
& = & \frac{1}{pv} + \frac{1}{qu} + \frac{1}{q'u} + \frac{1}{p'v}\\
& = & \frac{1}{v}\left(\frac{1}{p} + \frac{1}{p'}\right) +
\frac{1}{u}\left(\frac{1}{q} + \frac{1}{q'}\right)\\
& = & \frac{1}{v} \cdot 1 + \frac{1}{u} \cdot 1\\
& = & 1.
\end{eqnarray*}
Let us also observe that:
\begin{eqnarray}
\frac{\gamma^2}{pv} + \frac{\alpha^2}{qu} + \frac{1}{pquv} & = &
\frac{1}{uvr}. \label{J2}
\end{eqnarray}
Indeed, we have:
\begin{eqnarray*}
& \ & \frac{\gamma^2}{pv} + \frac{\alpha^2}{qu} + \frac{1}{pquv}\\
& = & \frac{p'}{u^2q'r'} \cdot \frac{1}{pv} + \frac{q'}{v^2p'r'}
\cdot \frac{1}{qu} + \frac{1}{pquv}\\
& = & \frac{1}{uvr'}\left(\frac{p'}{puq'} + \frac{q'}{qvp'} +
\frac{1}{pq} \cdot r'\right)\\
& = & \frac{1}{uvr'}\left[\frac{p'}{puq'} + \frac{q'}{qvp'} +
\frac{1}{pq} \cdot \left(\frac{p'}{u} + \frac{q'}{v}\right)\right]\\
& = & \frac{1}{uvr'}\left[\frac{p'}{puq'} + \frac{q'}{qvp'} +
\frac{p'}{pqu} + \frac{q'}{pqu}\right]\\
& = & \frac{1}{uvr'}\left[\frac{p'}{pu}\left(\frac{1}{q'} +
\frac{1}{q}\right) + \frac{q'}{qv}\left(\frac{1}{p'} +
\frac{1}{p}\right)\right]\\
& = & \frac{1}{uvr'}\left[\frac{p'}{pu} \cdot 1 + \frac{q'}{qv}
\cdot 1\right]\\
& = & \frac{1}{uvr'}\left[\frac{p/(p - 1)}{pu} + \frac{q/(q -
1)}{qv}
\right]\\
& = & \frac{1}{uvr'}\left[\frac{1}{u(p - 1)} + \frac{1}{v(q -
1)}\right]\\
& = & \frac{1}{uvr/(r - 1)}\cdot \frac{1}{r - 1}\\
& = & \frac{1}{uvr}.
\end{eqnarray*}
We go back to the denominator of formula (\ref{formula}) and apply
the Jensen inequality for the strictly concave downward function
$L(x) = \ln(x)$. From (\ref{J1}) we conclude that:
\begin{eqnarray*}
\ln\left(\frac{1}{pv} \cdot S + \frac{1}{qu} \cdot T + \beta^2 \cdot
1\right) & \geq & \frac{1}{pv}\ln(S) + \frac{1}{qu}\ln(T) +
\beta^2\ln(1)\\
& = & \ln\left(S^{1/(pv)}T^{1/{(qu)}}\right).
\end{eqnarray*}
Exponentiating both sides of this inequality and then rasing them to
the power $1/r'$, we obtain:
\begin{eqnarray}
\left(\frac{1}{pv} \cdot S + \frac{1}{qv} \cdot T + \beta^2 \cdot
1\right)^{1/r'} & \geq & S^{1/(pvr')}T^{1/(qur')}. \label{r'}
\end{eqnarray}
Formula (\ref{J2}) shows that in order to obtain a convex
combination in the sum:
\begin{eqnarray*}
& \ & \frac{\gamma^2}{pv}\cdot S + \frac{\alpha^2}{qu}\cdot T +
\frac{1}{pquv} \cdot ST
\end{eqnarray*}
we need first to multiply it by $K := uvr$. After doing this,
applying again the strict concavity of the function $\ln$, we
obtain:
\begin{eqnarray*}
\ln\left(K\frac{\gamma^2}{pv}\cdot S + K\frac{\alpha^2}{qu}\cdot T +
K\frac{1}{pquv} \cdot ST\right) & \geq & K\frac{\gamma^2}{pv}\ln(S)
+ K\frac{\alpha^2}{qu}\ln(T) + K\frac{1}{pquv}\ln(ST)\\
& = & \ln\left(S^{\frac{K\gamma^2}{pv} +
\frac{K}{pquv}}T^{\frac{K\alpha^2}{qu} + \frac{K}{pquv}}\right).
\end{eqnarray*}
This inequality is equivalent to:
\begin{eqnarray}
\left(\frac{\gamma^2}{pv}\cdot S + \frac{\alpha^2}{qu}\cdot T +
\frac{1}{pquv} \cdot ST\right)^{1/r} & \geq &
\frac{1}{K^{1/r}}S^{\frac{K\gamma^2}{pvr} +
\frac{K}{pquvr}}T^{\frac{K\alpha^2}{qur} +
\frac{K}{pquvr}}\nonumber\\
& = & \frac{1}{(uvr)^{1/r}}S^{\frac{u\gamma^2}{p} + \frac{1}{pq}}
T^{\frac{v\alpha^2}{q} + \frac{1}{pq}}. \label{r}
\end{eqnarray}
Going now back to the formula (\ref{formula}) and using the
inequalities (\ref{r'}) and (\ref{r}), we obtain:
\begin{eqnarray*}
& \ & \frac{p^{1/p}q^{1/q}}{r^{1/r}} \cdot \frac{s^{1/p}t^{1/q}}{(s
+ t + \beta^2)^{1/r'}(\gamma^2s + \alpha^2t + st)^{1/r}} \nonumber\\
& = & \frac{1}{v^{1/p}u^{1/q}r^{1/r}} \cdot
\frac{S^{1/p}T^{1/q}}{\left(\frac{1}{pv} \cdot S + \frac{1}{qv}
\cdot T + \beta^2 \cdot
1\right)^{1/r'}\left(\frac{\gamma^2}{pv}\cdot S +
\frac{\alpha^2}{qu}\cdot T + \frac{1}{pquv} \cdot ST\right)^{1/r}}\\
& \leq & \frac{1}{v^{1/p}u^{1/q}r^{1/r}} \cdot
\frac{S^{1/p}T^{1/q}}{S^{1/(pvr')}T^{1/(qur')}
\frac{1}{(uvr)^{1/r}}S^{(u\gamma^2)/p + 1/(pq)} T^{(v\alpha^2)/q +
1/(pq)}}.
\end{eqnarray*}
The exponent of $S$ in the denominator of the last fraction is:
\begin{eqnarray*}
\frac{1}{pvr'} + \frac{p'}{upq'r'} + \frac{1}{pq} & = &
\frac{1}{pr'}\left[\frac{1}{v} + \frac{p'}{uq'} + \frac{1}{q} \cdot
r'\right]\\
& = & \frac{1}{pr'}\left[\frac{1}{v} + \frac{p'}{uq'} +
\frac{1}{q}\left(\frac{p'}{u} + \frac{q'}{v}\right)\right]\\
& = & \frac{1}{pr'}\left[\frac{1}{v} +
\frac{p'}{u}\left(\frac{1}{q'} + \frac{1}{q}\right) + \frac{q'}{v}
\cdot \frac{1}{q}\right]\\
& = & \frac{1}{pr'}\left[\frac{1}{v} + \frac{p'}{u} \cdot 1 +
\frac{q'}{v}\left(1 - \frac{1}{q'}\right)\right]\\
& = & \frac{1}{pr'}\left[\frac{1}{v} + \frac{p'}{u} + \frac{q'}{v} -
\frac{1}{v}\right]\\
& = & \frac{1}{pr'} \cdot r'\\
& = & \frac{1}{p}.
\end{eqnarray*}
Similarly, the exponent of $T$ in the denominator of the same
fraction is $1/q$. Hence, for all $s$, $t > 0$, we have:
\begin{eqnarray}
\frac{p^{1/p}q^{1/q}}{r^{1/r}} \cdot \frac{s^{1/p}t^{1/q}}{(s + t +
\beta^2)^{1/r'}(\gamma^2s + \alpha^2t + st)^{1/r}} & \leq & v^{(1/r)
- (1/p)}u^{(1/r) - (1/q)} \cdot
\frac{S^{1/p}T^{1/q}}{S^{1/p}T^{1/q}} \nonumber\\
& = & v^{(1/r) - (1/p)}u^{(1/r) - (1/q)}. \label{supremum}
\end{eqnarray}
The equality in (\ref{supremum}) holds if and only if $S = T = 1$,
due to the strict concavity of the function $y = \ln(x)$. This is
equivalent to $s = 1/(pv)$ and $t = 1/(qu)$, since $S = (pv)s$ and
$T = (qu)t$. Thus, we have:
\begin{eqnarray*}
\sup_{s > 0, t > 0}\left\{\frac{p^{1/p}q^{1/q}}{r^{1/r}} \cdot
\frac{s^{1/p}t^{1/q}}{(s + t + \beta^2)^{1/r'}(\gamma^2s + \alpha^2t
+ st)^{1/r}}\right\} & = & v^{(1/r) - (1/p)}u^{(1/r) - (1/q)}.
\end{eqnarray*}
Going back to the inequality (\ref{L2}), we conclude that:
\begin{eqnarray*}
& \ & \left\{\int_{{\mathbb
R}^d}\left|\left[\Gamma\left(\frac{1}{\sqrt{u}}\right)\varphi
\diamond \Gamma\left(\frac{1}{\sqrt{v}}\right)\psi\right]
\left(\frac{x}{\sqrt{uv}}\right)e^{-\frac{\langle x, x
\rangle}{2uv}\cdot\frac{1}{r}}\right|^rd_Nx\right\}^{1/r}\\
& \leq & \left\{\int_{{\mathbb R}^d}\left[\int_{{\mathbb R}^d}|f(x -
y)| \cdot |g(y)| \cdot e^{-\frac{1}{2p'q'r'}\left\langle
\frac{q'}{v}x - r'y, \frac{q'}{v}x -
r'y\right\rangle}d_Ny\right]^rd_Nx\right\}^{1/r}\\
& \leq & \sqrt{v^{\frac{d}{r} - \frac{d}{p}}u^{\frac{d}{r} -
\frac{d}{q}}}
\||f|\|_p\||g|\|_q\\
& = & v^{\frac{d}{2r} - \frac{d}{2p}}u^{\frac{d}{2r} -
\frac{d}{2q}}\left[\int_{{\mathbb
R}^d}\left|\varphi\left(\frac{x}{\sqrt{v}}\right)\right|^pe^{-\frac{\langle
x, x \rangle}{2v}}d_Nx\right]^{1/p}\left[\int_{{\mathbb
R}^d}\left|\psi\left(\frac{x}{\sqrt{u}}\right)\right|^qe^{-\frac{\langle
x, x \rangle}{2u}}d_Nx\right]^{1/q}.
\end{eqnarray*}
This inequality is equivalent to:
\begin{eqnarray*}
& \ & (uv)^{-\frac{d}{2r}}\left\{\int_{{\mathbb
R}^d}\left|\left[\Gamma\left(\frac{1}{\sqrt{u}}\right)\varphi
\diamond \Gamma\left(\frac{1}{\sqrt{v}}\right)\psi\right]
\left(\frac{x}{\sqrt{uv}}\right)\right|^re^{-\frac{\langle x, x
\rangle}{2uv}}d_Nx\right\}^{1/r}\\
& \leq & v^{-\frac{d}{2p}}\left[\int_{{\mathbb
R}^d}\left|\varphi\left(\frac{x}{\sqrt{v}}\right)\right|^pe^{-\frac{\langle
x, x
\rangle}{2v}}d_Nx\right]^{1/p}u^{-\frac{d}{2q}}\left[\int_{{\mathbb
R}^d}\left|\psi\left(\frac{x}{\sqrt{u}}\right)\right|^qe^{-\frac{\langle
x, x \rangle}{2u}}d_Nx\right]^{1/q}.
\end{eqnarray*}
Making now the changes of variable $x' := x/\sqrt{uv}$ in the left,
and $x_1 := x/\sqrt{v}$ and $x_2 := x/\sqrt{u}$ in the right, and
moving back to the Gaussian norms, we obtain:
\begin{eqnarray*}
\left\|\Gamma\left(\frac{1}{\sqrt{u}}\right)\varphi \diamond
\Gamma\left(\frac{1}{\sqrt{v}}\right)\psi\right\|_r & \leq &
\|\varphi\|_p \cdot \|\psi\|_q.
\end{eqnarray*}
To prove the inequality in the infinite dimensional case we proceed
in the following way. Let $H$ be a separable Gaussian Hilbert space.
Let $\{e_n\}_{n \geq 1}$ be an orthonormal basis of centered
Gaussian random variables from $H$. For all $d \geq 1$, let
\begin{eqnarray*}
H_d & := & {\mathbb C}e_1 \oplus {\mathbb C}e_2 \oplus \cdots \oplus
{\mathbb C}e_d.
\end{eqnarray*}
Let ${\mathcal F}_d := {\mathcal F}(H_d)$, i.e., the smallest
sigma--algebra with respect to which $e_1$, $e_2$, $\dots$, $e_d$
are measurable. If $\varphi \in L^{p}(\Omega$, ${\mathcal F}(H)$,
$P)$ and $\psi \in L^{q}(\Omega$, ${\mathcal F}(H)$, $P)$, and if we
denote the conditional expectations of $\varphi$ and $\psi$, with
respect to ${\mathcal F}_d$, by $\varphi_d$ and $\psi_d$,
respectively, i.e., $\varphi_d := E[\varphi \mid {\mathcal F}_d]$
and $\psi_d := E[\psi \mid {\mathcal F}_d]$, then it is not hard to
see that:
\begin{eqnarray}
E[\Gamma(1/\sqrt{u})\varphi \mid {\mathcal F}_d] & = &
\Gamma(1/\sqrt{u})\varphi_d,
\end{eqnarray}
\begin{eqnarray}
E[\Gamma(1/\sqrt{v})\psi \mid {\mathcal F}_d] & = &
\Gamma(1/\sqrt{v})\psi_d,
\end{eqnarray}
and
\begin{eqnarray}
E[\Gamma(1/\sqrt{u})\varphi \diamond \Gamma(1/\sqrt{v})\psi \mid
{\mathcal F}_d] & = & \Gamma(1/\sqrt{u})\varphi_d \diamond
\Gamma(1/\sqrt{v})\psi_d.
\end{eqnarray}
Since $\{{\mathcal F}_d\}_{d \geq 1}$ is an increasing family of
sigma--algebras and the sigma--algebra generated by them is
${\mathcal F}(H)$, using the Martingale Convergence Theorem we
conclude that:
\begin{eqnarray*}
E[\Gamma(1/\sqrt{u})\varphi \diamond \Gamma(1/\sqrt{v})\psi \mid
{\mathcal F}_d] & \to & \Gamma(1/\sqrt{u})\varphi \diamond
\Gamma(1/\sqrt{v})\psi,
\end{eqnarray*}
\begin{eqnarray*}
E[\varphi \mid {\mathcal F}_d] & \to & \varphi,
\end{eqnarray*}
and
\begin{eqnarray*}
E[\psi \mid {\mathcal F}_d] & \to & \psi,
\end{eqnarray*}
as $d \to \infty$, both almost surely and in $L^1(\Omega$,
${\mathcal F}(H)$, $P)$. Using now the fact that the result is true
in the finite dimensional case and Fatou's Lemma as before, we
conclude that:
\begin{eqnarray*}
\left\|\Gamma(1/\sqrt{u})\varphi \diamond
\Gamma(1/\sqrt{v})\psi\right\|_r & \leq &
\|\varphi\|_p\cdot\|\psi\|_q.
\end{eqnarray*}
\end{proof}

\end{document}